\documentclass[12pt]{article}
\usepackage[utf8]{inputenc}
\usepackage{graphicx}
\usepackage{color}
\usepackage{times}
\usepackage[hidelinks]{hyperref}
\usepackage{enumerate,latexsym}
\usepackage{latexsym}
\usepackage{amsmath,amssymb}
\usepackage{graphicx}

\usepackage{amsmath,amsthm,amsfonts,amssymb}
\usepackage{graphicx}
\usepackage{dsfont}

\makeatletter
\def\namedlabel#1#2{\begingroup
 #2%
 \def\@currentlabel{#2}%
 \phantomsection\label{#1}\endgroup
}
\makeatother

\theoremstyle{plain}
\newtheorem*{theorem*}{Theorem}
\newtheorem*{thmex*}{Theorem~\ref{example}}
\newtheorem*{thmasymp*}{Theorem~\ref{thmAsymp}}
\newtheorem{theorem}{Theorem}[section]

\newtheorem{claim}[theorem]{Claim}

\newtheorem{lemma}[theorem]{Lemma}

\newtheorem{remark}[theorem]{Remark}

\newtheorem{example}[theorem]{Example}
\newtheorem*{eg*}{Example}

\theoremstyle{definition}
\newtheorem{definition}[theorem]{Definition}

\newcommand{\es}{\emptyset}
\newcommand{\R}{\mathbb{R}}

\newcommand{\sn}[1]{{\mathbb{S}^{#1}}}
\newcommand{\hn}[1]{{\mathbb{H}^{#1}}}

\newcommand{\G}{\Gamma}
\newcommand{\ben}{\begin{enumerate}}
\newcommand{\een}{\end{enumerate}}
\newcommand{\wt}{\widetilde}
\newcommand{\g}{\gamma}
\newcommand{\be}{\beta}

\renewcommand{\a}{\alpha}

\newcommand{\cB}{{\mathcal B}}

\newcommand{\rth}{\R^3}

\newcommand{\D}{\Delta}

\newcommand{\cQ}{{\mathcal Q}}

\newcommand{\cS}{{\mathcal S}}

\newcommand{\abs}[1]{\vert #1 \vert}

\newcommand{\ed}{\end{document}}
\renewcommand{\S}{\Sigma}

\newcommand{\ol}{\overline}
\newcommand{\wh}{\widehat}
\definecolor{rrr}{rgb}{.9,0,.1}

\definecolor{rr}{rgb}{.8,0,.3}
\definecolor{pp}{rgb}{.5,0,.7}

\usepackage{pdfsync}
\begin{document}

\title{Modifications Preserving Hyperbolicity of Link Complements}

\author{Colin Adams
\and William H. Meeks III
\and \'Alvaro K. Ramos \thanks{The second and third
authors were partially supported
by CNPq - Brazil, grant no. 400966/2014-0.}.}

\maketitle

\begin{abstract}
Given a link in a 3-manifold such that the complement
is hyperbolic, we provide two modifications to the link,
called the chain move and the switch move, that preserve
hyperbolicity of the complement, with only a relatively
small number of manifold-link pair exceptions, which are
also classified. These
modifications  provide a substantial increase in the number of known
hyperbolic links in the 3-sphere and other 3-manifolds.\end{abstract}

\section{Introduction.}

Thurston proved that every knot in the 3-sphere $S^3$ is either a torus knot,
a satellite knot or a hyperbolic knot, by which we mean that its complement
in $S^3$  admits a
complete hyperbolic metric.
By the Mostow-Prasad Rigidity Theorem, the complement of a  hyperbolic knot in $S^3$ has
a unique hyperbolic metric, which must have finite volume; hence, a
hyperbolic knot in $S^3$ has associated to it a
well-defined set of hyperbolic invariants such as volume, cusp volume,
cusp shape, etcetera.
More generally, Thurston proved that a link in a closed,
orientable 3-manifold
has hyperbolic complement (necessarily of finite volume) if and only if
the exterior of the link contains no properly embedded essential
disks, spheres, tori or annuli, terms that are described in Definition~\ref{Def:2.1}.

One would like to be able to identify link complements that
satisfy Thurston's criteria, and that therefore possess
a hyperbolic metric. In~\cite{mena}, Menasco proved that every
non-2-braid prime alternating link in $S^3$ is hyperbolic.
In~\cite{A1}, Adams extended this result to augmented alternating links,
where additional non-parallel trivial components
wrapping around two adjacent strands in the alternating projection
were added to the link.
These additional components bound twice-punctured disks, which
are totally geodesic in the hyperbolic structure of
the complement. By~\cite{ad1}, the link complement can be cut
open along such a twice-punctured disk, twisted a
half-twist and re-glued to obtain another hyperbolic link
complement, with identical volume. This operation adds one crossing
to the link projection. In many hyperbolic link complements,
twice-punctured disks are particularly useful,
because they are totally geodesic; see for instance~\cite{pur1}
and the references therein.

We consider two moves that one can perform on a link in a
3-manifold with hyperbolic complement. The first move we
consider is called the {\it chain move}. Here, we start
with a trivial component bounding a twice-punctured disk
in a ball $\cB$ as in Figure~\ref{chainlemma1}, and we replace the
tangle on the left with the tangle on the right
in Figure~\ref{chainlemma1}, where $k$ is any integer. Assuming that
the rest of the manifold outside $\cB$ is not
the complement of a rational tangle in a 3-ball (see Chapter~2
of~\cite{Adbook} for this definition), the result is hyperbolic.

\begin{figure}[htpb]
\begin{center}
\includegraphics[width=0.9\textwidth]{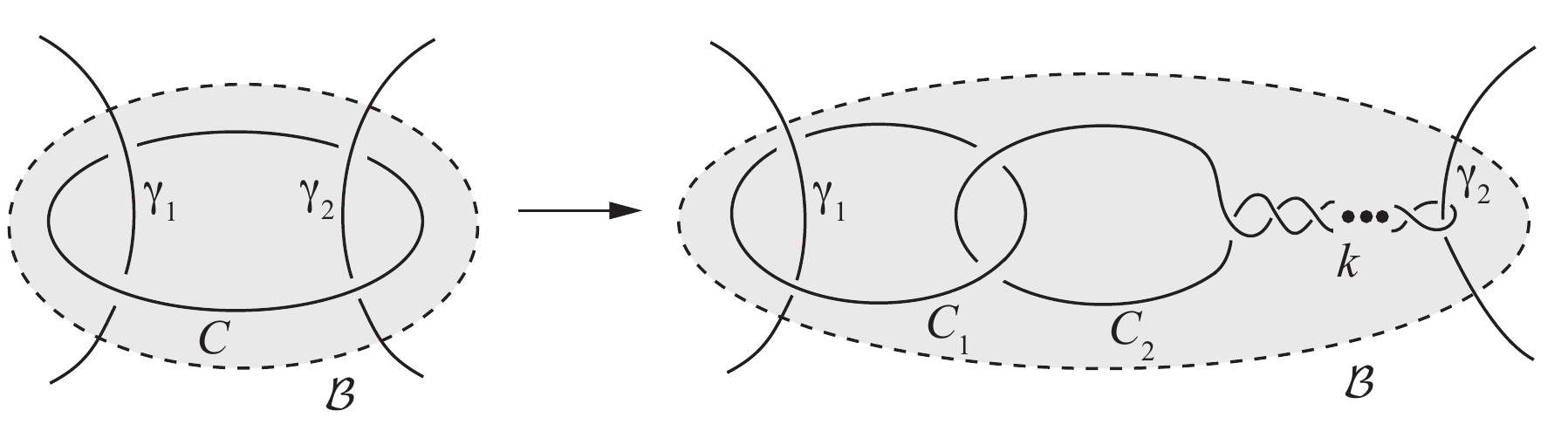}
\caption{Replacing (a) with
(b) preserves hyperbolicity of the complement.}
\label{chainlemma1}
\end{center}
\end{figure}

There are counterexamples to extending the result to the case where the
manifold outside $\cB$ is a
rational tangle complement in a 3-ball as demonstrated by the
hyperbolic link in the 3-sphere appearing
in Figure~\ref{chainmovecounterexample}. When the chain move is applied
with $k = 3$, the
resultant 3-component link is $6^3_3$ in Alexander-Rolfsen notation,
which is not hyperbolic.
However, in Lemma~\ref{rational}, we delineate explicitly the
only possible exceptions.

\begin{figure}[htpb]
\begin{center}
\includegraphics[width=1.0\textwidth]{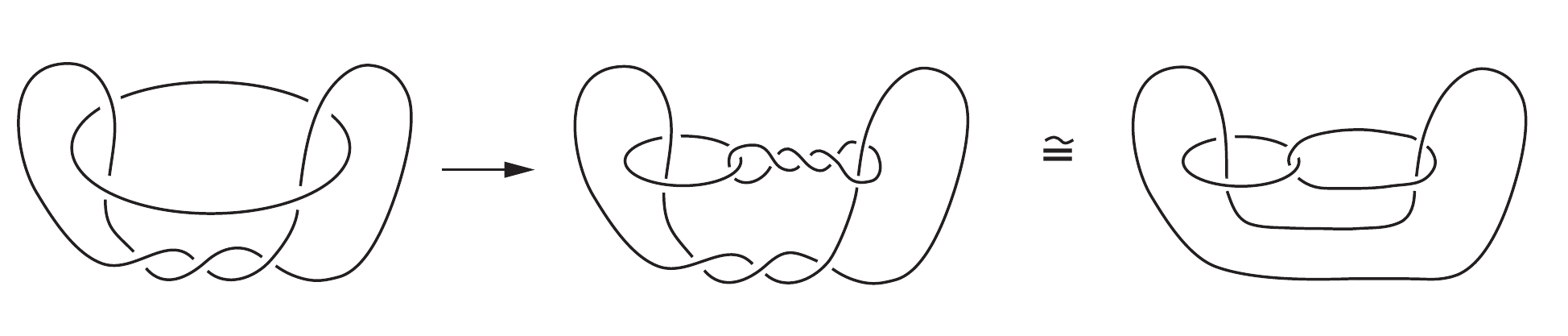}
\caption{Applying the chain move to this hyperbolic link with $k = 3$
yields the non-hyperbolic link complement $6^3_3$.}
\label{chainmovecounterexample}
\end{center}
\end{figure}

The second move is called the {\it switch move}. Suppose we have a
3-manifold $M$ and a link $L$ in $M$
with hyperbolic complement. Let $\alpha$ be an embedded arc that runs
from $L$ to $L$ with interior that
is isotopic to an embedded geodesic in the complement, as in
Figure~\ref{switchbefore}.

\begin{figure}
\centering
\includegraphics[width=0.4\textwidth]{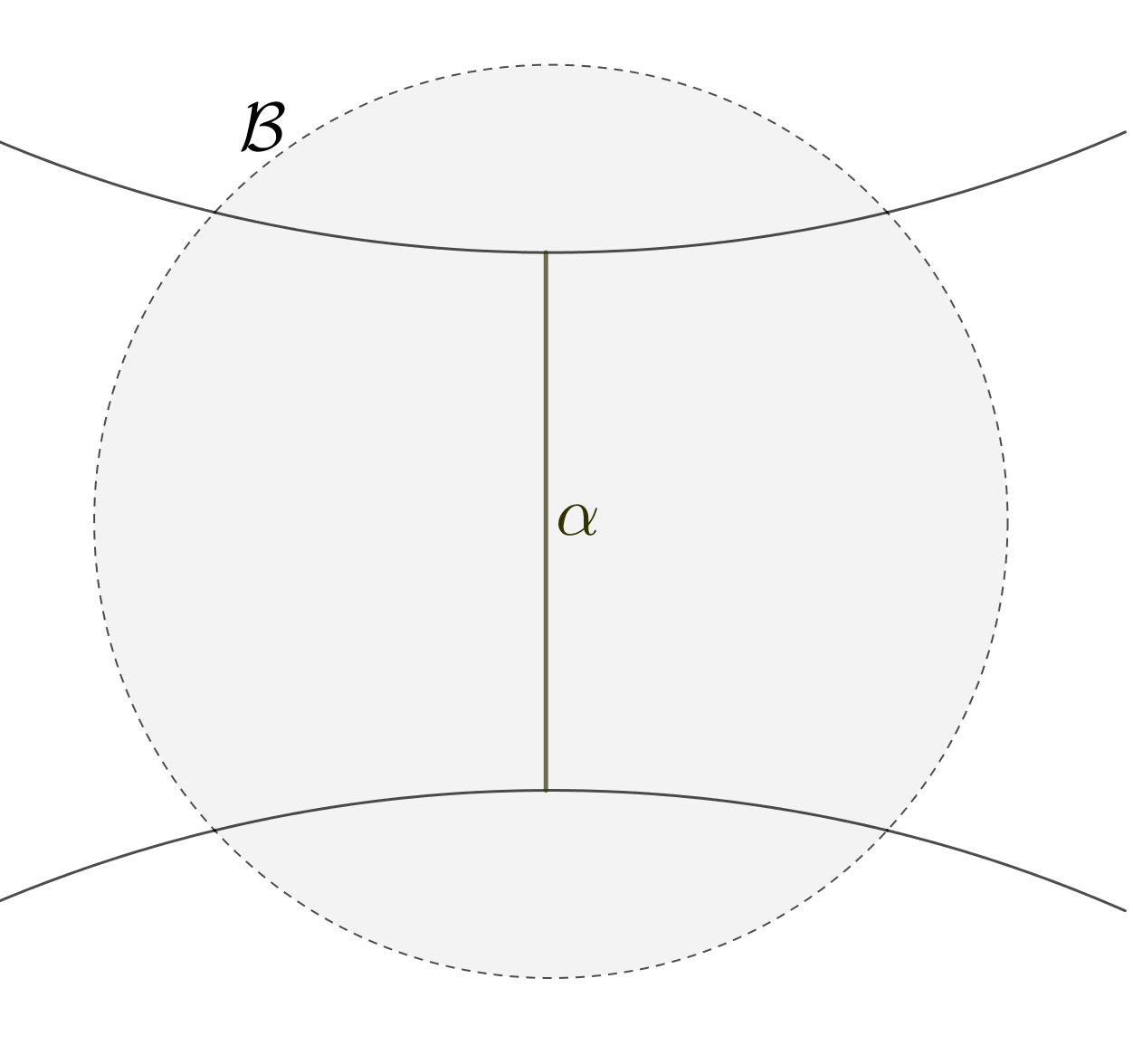}
\caption{The trace of a geodesic $\alpha$ of $(M\setminus L, h)$
connects one or two components of $L$ to one another, and a neighborhood
$\cB$ of $\a$ intersects $L$ in two arcs.\label{switchbefore}}
\end{figure}

Such a geodesic always exists since we could take one with minimal
length outside fixed cusp boundaries.
We consider the possibility that the arc runs from one component of $L$
back to the same component or
from one component to a second component. Let $\cB$ be a neighborhood
of $\alpha$. Then $\cB$ intersects
$L$ in two arcs, as in Figure~\ref{figaugmented}~(a). The switch move
allows us to surger the link and
add in a trivial component as in Figure~\ref{figaugmented}~(b) while
preserving hyperbolicity.

\noindent
\begin{figure}
\noindent \quad
\quad\begin{minipage}{0.42\textwidth}
\begin{center}
\includegraphics[width=0.99\textwidth]{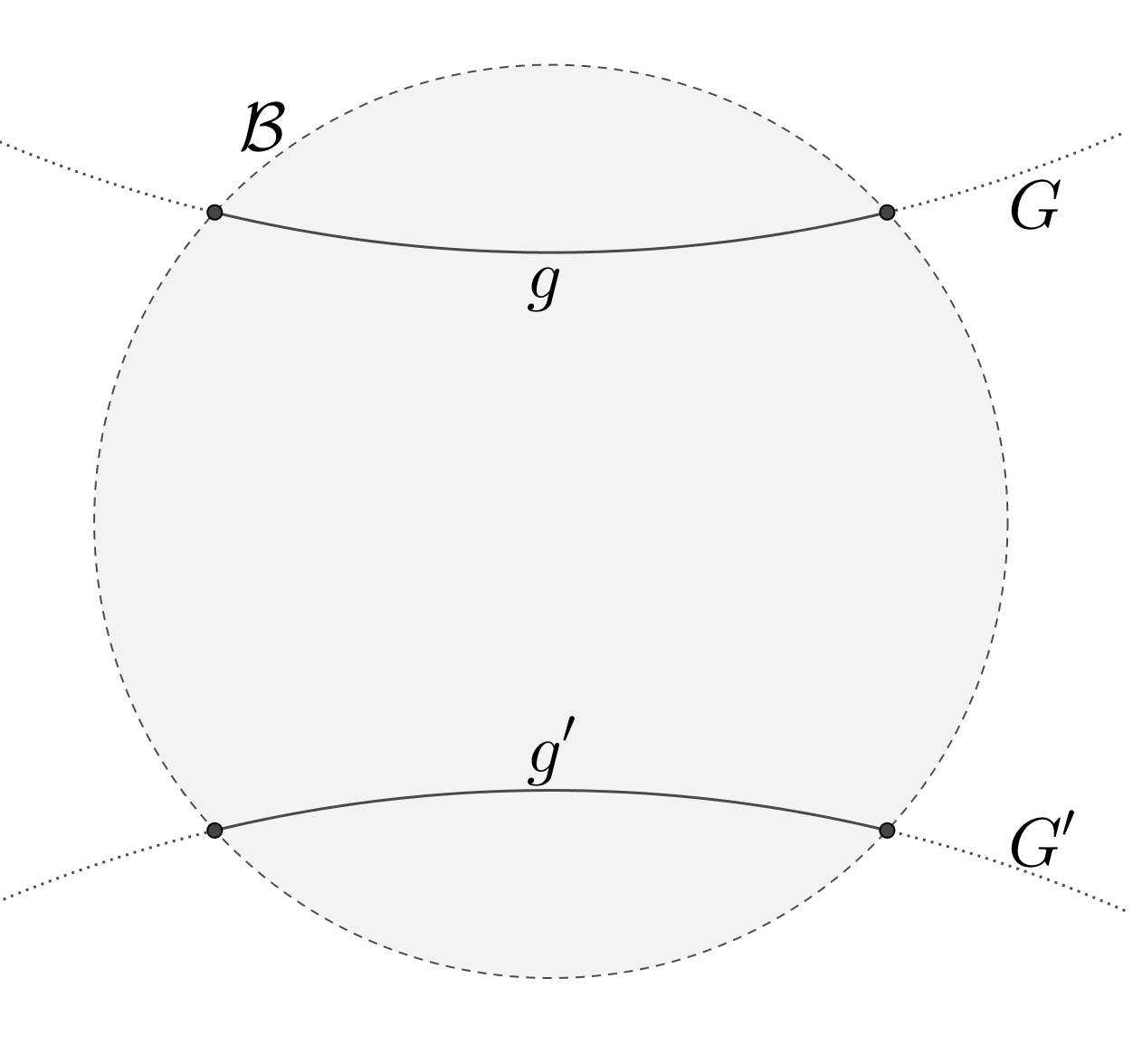}

(a)
\end{center}
\end{minipage}\hfill
$\longrightarrow$\hfill
\begin{minipage}{0.42\textwidth}
\begin{center}
\includegraphics[width=0.99\textwidth]{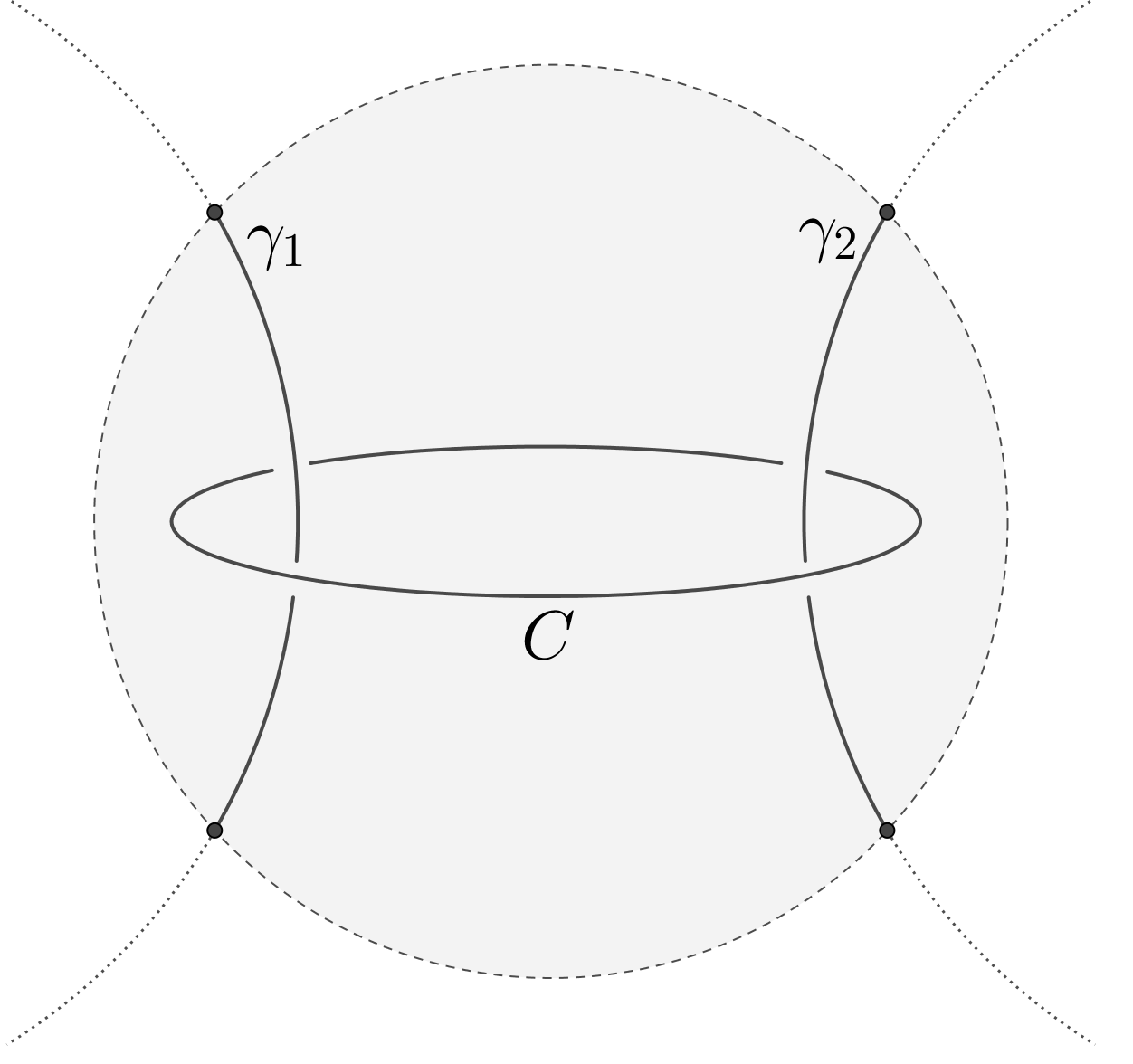}

(b)
\end{center}
\end{minipage}
\quad \quad
\caption{The switch move replaces the
arcs $g$ and $g'$ by the tangle $\g_1\cup~\g_2\cup C$.}
\label{figaugmented}
\end{figure}

\begin{remark}\label{notwelldefined}{\em
Note that the projection depicted in Figure~\ref{switchbefore}
is not well-defined,
since if the two arcs are skew inside the ball, there are two different
projections, depending on point of view.
So in fact, for each such geodesic $\alpha$, there are two switch moves
possible. This is equivalent to
cutting along the twice-punctured disk $D$ bounded by $C$ and twisting
a half-twist in either direction
before regluing. Once we prove that the switch move depicted in
Figure~\ref{figaugmented} preserves hyperbolicity,
the hyperbolicity of the half-twisted version follows immediately from
the previously mentioned results of~\cite{ad1}.
Further twists give link complements homeomorphic to the original or the half-twisted version.} 
\end{remark}

These moves show that many additional link complements in 3-manifolds
are hyperbolic.
In~\cite{amr1}, the chain move and the switch move, together with the
related {\it switch move gluing operation} described in Section~\ref{secGlue},
are utilized in the proof that, for any given
surface $S$ of finite topology
and negative Euler characteristic and any $H\in[0,1)$, there exists a
proper, totally umbilic embedding
of $S$ into some hyperbolic 3-manifold of finite volume with
image surface having constant mean curvature $H$.

In~\cite{SMALL2018}, the chain move is used in the proof that a virtual
link obtained by taking a reduced classical prime alternating
link projection and changing one crossing
to be virtual yields a non-classical virtual link.
See that paper for details.

We can also use the chain move and the switch move to obtain
straightforward proofs of hyperbolicity of well-known classes of links.

\begin{example}{\em
We can show that every chain link of five or more components, no matter
how twisted, is hyperbolic.
This was first proved in~\cite{NR} using explicit hyperbolic structures
for manifolds covered by these link complements.

Start with the alternating 4-chain, known to be hyperbolic by Menasco's
work in~\cite{mena}.
Then apply the chain move repeatedly. This proves hyperbolicity of any
chain link of five
or more components with an arbitrary amount of twisting in the chain.}
\end{example}

We note that the chain and switch moves apply more broadly than is apparent from
Figures~\ref{chainlemma1} and \ref{figaugmented}.
In the case of the chain move, instead of specifying a hyperbolic link
complement $M \setminus L$, we can start with a cusped hyperbolic 3-manifold $M'$
containing a two-sided essential embedded thrice-punctured sphere $S$.
Treating two of the boundary curves on the cusps as the meridianal punctures
of the disk in Figure~\ref{figaugmented} and the third as the longitudinal
boundary of the disk, we can apply the chain move, removing the cusp that contains
the longitude by doing a Dehn filling along a curve that crosses the longitude
once and adding in the additional two components within a neighborhood of $S$.
In the case that two of the boundaries of $S$ are on the same cusp, they must
play the role of the meridianal punctures. (Note that a two-sided thrice
punctured sphere cannot have all three boundaries on the same cusp.)

In the case of the switch move, we can again begin with a cusped
hyperbolic 3-manifold $M'$. For two cusps connected by an embedded geodesic,
we can choose a nontrivial simple closed curve on each torus corresponding to
each cusp. Then by Dehn filling along those curves we obtain a 3-manifold $M$
for which $M'$ is a link complement and the switch move applies.

 The same procedure holds for a geodesic from a cusp back to the same cusp, and a
 specification of a nontrivial simple closed curve on the torus corresponding
 to the cusp, two copies of which play the role of the meridians around $\g_1$
 and $\g_2$. Note that when applied to a link complement $M \setminus L$,
 but with a choice of curve other than meridians, the end result is not a new
 link complement in the same manifold.

 Finally, we point out that there is a variant of the chain move called
 the {\it augmented chain move} as in Figure~\ref{augmentedchainmove} wherein
 the two new components of the chain move are added in but the previous trivial
 component is not removed. We prove here that this move also preserves hyperbolicity.

\begin{figure}[htpb]
\begin{center}
\includegraphics[width=.8\textwidth]{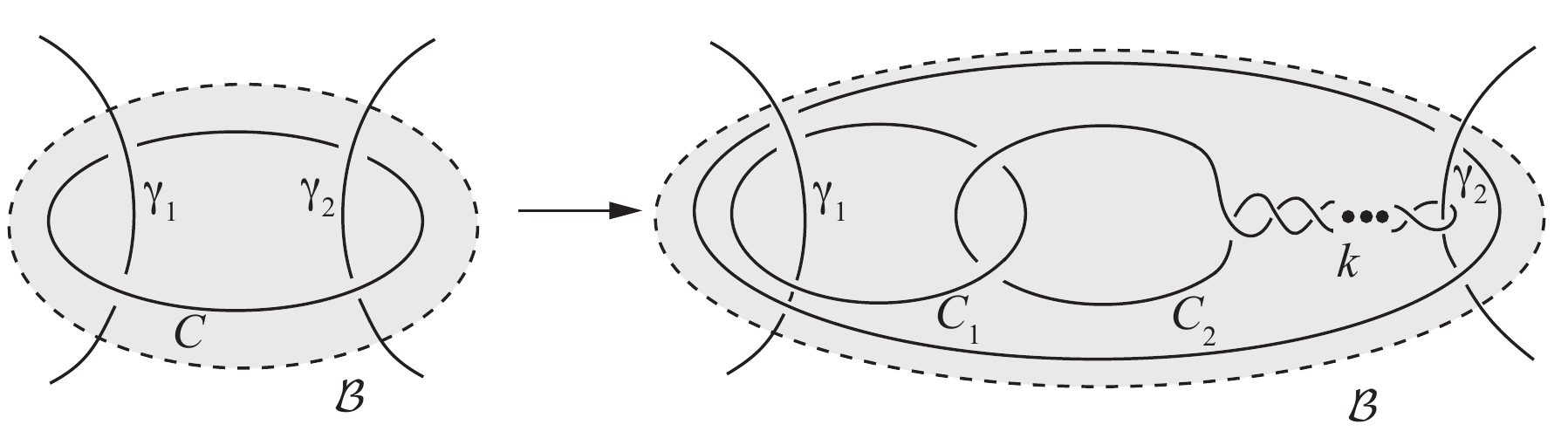}
\caption{The augmented chain move.}
\label{augmentedchainmove}
\end{center}
\end{figure}

To see this, we consider the link appearing in Figure~\ref{5chain}, which is a twisted five-chain.
\begin{figure}[htpb]
\begin{center}
\includegraphics[width=.4\textwidth]{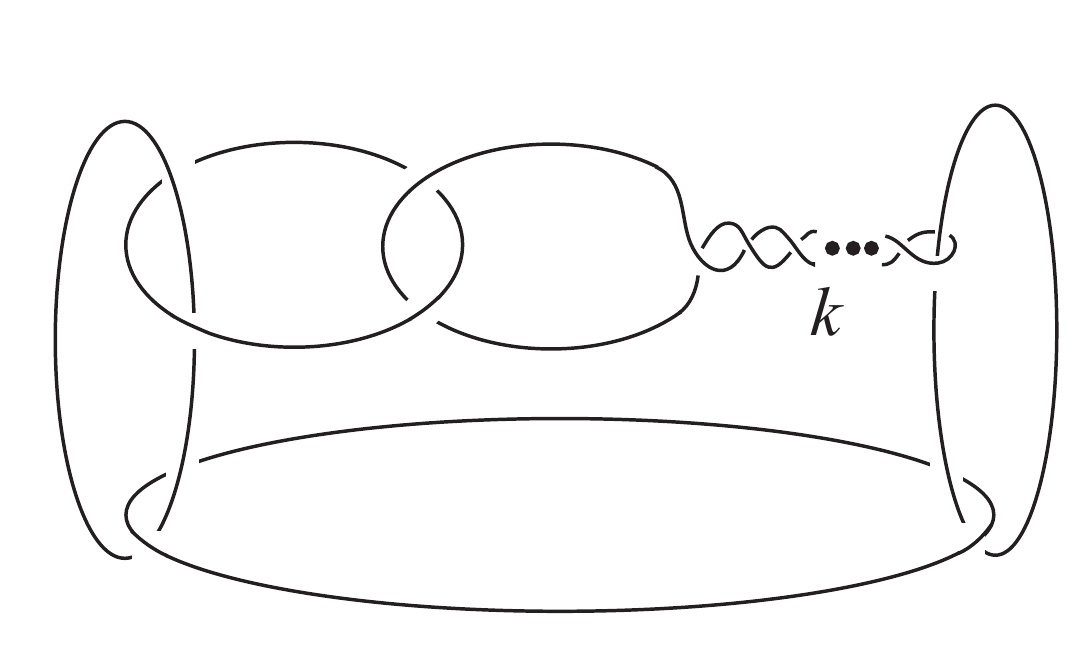}
\caption{All 5-chains are hyperbolic.}
\label{5chain}
\end{center}
\end{figure}

All five-chains are hyperbolic, as we just proved, so it has a hyperbolic
complement. Now we apply the idea of a walnut as in~\cite{Adams7}. We can
cut the manifold $M$ open along the twice-punctured disk bounded by $C$,
cut the 5-chain link complement open along the twice punctured disk $E$,
and then glue copies of the twice-punctured disks to one another to insert
the cut-open link complement into $M$. As in~\cite{ad1}, since a twice-punctured
disk is totally geodesic with a rigid unique hyperbolic structure, the gluings
are isometries and the resulting manifold is hyperbolic with volume the sum of
the volumes of the two manifolds.

Next, we explain the organization of the paper.
First, we remark that it suffices to demonstrate our results when the ambient
manifold is orientable. This property is proved by
showing that the oriented cover of a related non-orientable
link complement admits a hyperbolic
metric and then one applies the  Mostow-Prasad Rigidity Theorem to conclude
that the associated order-two covering transformation
is an isometry, which in turn  implies that the hyperbolic metric
on the oriented covering descends.
In Section~\ref{secprelim}, we present some of the background material
necessary to the proofs of our main results in the orientable setting.
In Section~\ref{secChain}, we prove the Chain Move Theorem, stated
there as Theorem~\ref{chain}. In Section~\ref{secSwitch},
we prove the Switch Move Theorem, see Theorem~\ref{switch}.
In Section~\ref{secGlue} we prove the Switch Move Gluing Operation,
which constructs from a pair of hyperbolic 3-manifolds of finite
volume, an infinite number of new hyperbolic 3-manifolds of finite
volume, see Theorem~\ref{thmoperation}.
In the appendix at the end of the paper,
we classify the exceptional links for which the chain move
fails to produce hyperbolic 3-manifolds of finite volume.

\section{Preliminaries.}\label{secprelim}
In this section, we recall some definitions and results that
are needed to understand hyperbolic 3-manifolds of finite volume
and certain embedded surfaces in such ambient spaces.
Our first goal is to understand the statement of
Thurston's hyperbolization theorem
in our setting. Before stating this result, we first explain some
of the definitions and notations we use.
Throughout this discussion, $P$ will denote
a connected, orientable, compact 3-manifold with
nonempty boundary $\partial P$
consisting of tori and such that $P$ is not the product
of a two-torus with an interval, and ${\rm int}(P)$ will denote the
interior of $P$.
Moreover, a {\em surface $\S$ in $P$} means
a properly embedded surface $\S\subset P$, i.e., $\S$ is embedded
in $P$ with $\partial \S = \S \cap \partial P$.

\begin{definition} \label{Def:2.1} \
\begin{enumerate}[1.]
\item Given a surface $\S$ in $P$, a {\em compression disk for
$\S$} is a disk $E\subset P$ with $\partial E = E\cap \S$
such that $\partial E$ is
homotopically nontrivial in $\S$. If $\S$ does not admit
any compression disk, we say $\S$ is {\em incompressible}.
\item Given a surface $\S$ in $P$, a
{\em boundary-compression disk for
$\S$} is a disk $E\subset P$ with
$\partial E = E\cap (\S \cup \partial P)$
such that $\partial E = \alpha \cup \beta$, where $\alpha$ and $\beta$
are arcs intersecting only
in their endpoints such that $\alpha = E \cap \S$ and
$\beta = E \cap \partial P$ and $\alpha$ does
not cut a disk from $\S$. If $\S$ does not admit
any boundary-compression disk, we say $\S$ is
{\em boundary-incompressible}.
\item A torus $T$ in $P$ is {\em boundary parallel}
if $T$ is isotopic to a boundary component of $P$.
\item An annulus $A$ in $P$ is {\em boundary parallel}
if $A$ is isotopic, relative to $\partial A$,
to an annulus $A'\subset \partial P$.
\item A sphere $S$ in $P$ is {\em essential} if $S$ does not
bound a ball in $P$.
\item A disk $E$ in $P$ is {\em essential}
if $\partial E$ is homotopically nontrivial in $\partial P$.
\item A torus $T$ (respectively an annulus $A$) is {\em essential}
in $P$ if $T$ (resp. $A$) is incompressible and not boundary parallel.
\end{enumerate}
\end{definition}

Using the above definitions,
Thurston's hyperbolization theorem
implies that a connected, orientable,
noncompact 3-manifold $N$ admits
a hyperbolic metric of finite volume if and only if
$N$ is diffeomorphic to ${\rm int}(P)$ as above and
there are no essential spheres, disks, tori or annuli properly
embedded in $P$. In this case, we shall say that $N$ is
{\it hyperbolic}. When a link $L$ in a 3-manifold $M$ has
hyperbolic complement, we will say either $M \setminus L$
is hyperbolic, or $L$ is hyperbolic.

A useful fact is that if $\alpha$ is an arc with endpoints in a link
$L$ in a 3-manifold $M$
such that $\alpha$ corresponds to a geodesic in the hyperbolic link
complement $M \setminus L$,
then $\alpha$ cannot be homotoped through $M \setminus L$ into
$L$ while fixing its endpoints on $L$.
This follows from the fact any such geodesic will lift to
geodesics connecting distinct horospheres
in the universal cover $\mathbb{H}^3$, whereas an arc that is
homotopic into $L$ will lift to arcs,
each of which connects one and the same horosphere.

In the case that a manifold $M$ has no essential disks, we say it is
{\em boundary-irreducible}.
In the case that a manifold $M$ has no essential spheres, we say it is
{\em irreducible}. Note that
if $M$ has only toroidal boundaries and it is not a solid torus, which is
the situation we will consider, irreducibility implies
boundary-irreducibility. This is because if there exists an essential
disk $D$ with boundary
in a torus $T\subset \partial M$, then $\partial N(D \cup T)$ is a sphere
which must bound a ball to the non-$D$ side,
implying $M$ is a solid torus. Here and elsewhere, $N(G)$ denotes a
regular neighborhood of a set $G\subset M$.

Given an annulus $A$ properly embedded in an irreducible manifold
$M$ with toroidal boundary,
we note that if $A$ is boundary-compressible, it is boundary-parallel.
This follows because
we can surger the annulus along the boundary-compressing disk to obtain
a properly embedded disk $D$,
with trivial boundary on $\partial M$. Then $\partial D$ bounds a disk
$D'$ in $\partial M$,
and $D \cup D'$ is a sphere bounding a ball. This allows us to isotope
$A$ relative $\partial A$ into $\partial M$.

\section{The Chain Move Theorem}\label{secChain}

Let $L$ be a hyperbolic link in a 3-manifold $M$ and let
$\cB\subset M$ be a ball in $M$ that intersects $L$ as in
Figure~\ref{chainlemma1}~(a). In this section we prove the Chain
Move Theorem, as stated by Theorem~\ref{chain} below.
This proof breaks up into two cases
depending on whether or not the pair
$(M\setminus \cB, L\setminus \cB)$ is a rational tangle in
a 3-ball, see~\cite[Chapter~2]{Adbook} for this definition
and for the representation of a rational tangle
by a sequence of integers. Since the proof of the chain move
in this specific setting uses a specialized knowledge of
classical knot theory, it will be presented in the
\hyperref[rationaltangles]{Appendix} of this paper.

\begin{theorem}[{Chain Move Theorem}]\label{chain}
Let $L$ be a link in a 3-manifold $M$
such that the link complement $M\setminus L$ admits a
complete hyperbolic metric of finite volume.
Suppose that there is a sphere $\cS$ in $M$ bounding a ball $\cB$
that intersects $L$ as in Figure~\ref{chainlemma1}~(a).
Let $L'$ be the resulting link obtained by replacing
$L \cap \cB$ by the components as appear in
Figure~\ref{chainlemma1}~(b). Then if
$(M \setminus \cB, L \setminus (\cB \cap L))$
is not any of the rational tangles $-k$, $-(k+1)$, or $-2 -k$ in a 3-ball,
then $M \setminus L'$ admits a complete hyperbolic metric of
finite volume.
\end{theorem}

In Figure~\ref{excludedlinks}(a), we see the new link components that are
inserted into the ball $\cB$. In Figures~\ref{excludedlinks}(b), (c) and (d),
we see, for any fixed integer $k$, the three cases of rational tangles in the
exterior 3-ball that do not yield a hyperbolic link complement.

\begin{figure}[htpb]
\begin{center}
\includegraphics[width=.8\textwidth]{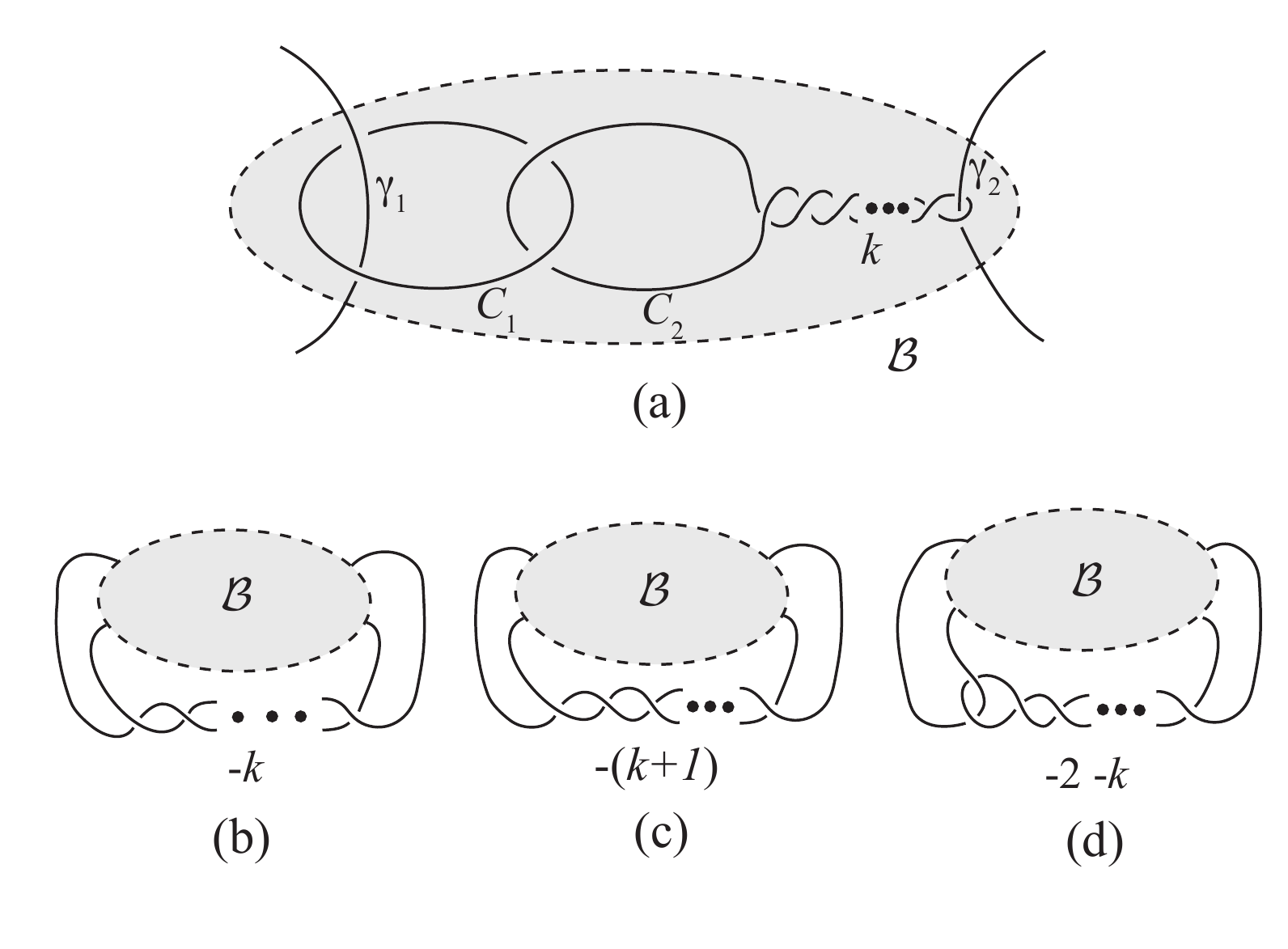}
\caption{The link components we are inserting in $\cB$ and the three rational
tangles in an exterior ball that do not generate a hyperbolic link complement.}
\label{excludedlinks}
\end{center}
\end{figure}

\begin{remark}{\em
The crossings around the single trivial component need
not be non-alternating for Theorem~\ref{chain} to apply.
If the crossings
alternate (as shown in Figure~\ref{reidemeister}~(a)),
we could add a crossing to $\gamma_2$ and
work in a sub-ball as in Figure~\ref{reidemeister}~(b)
so that the crossings are those shown in
Figure~\ref{chainlemma1}(a).

\noindent
\begin{figure}[ht]
\noindent \quad
\quad\begin{minipage}{0.38\textwidth}
\begin{center}
\includegraphics[width=0.99\textwidth]{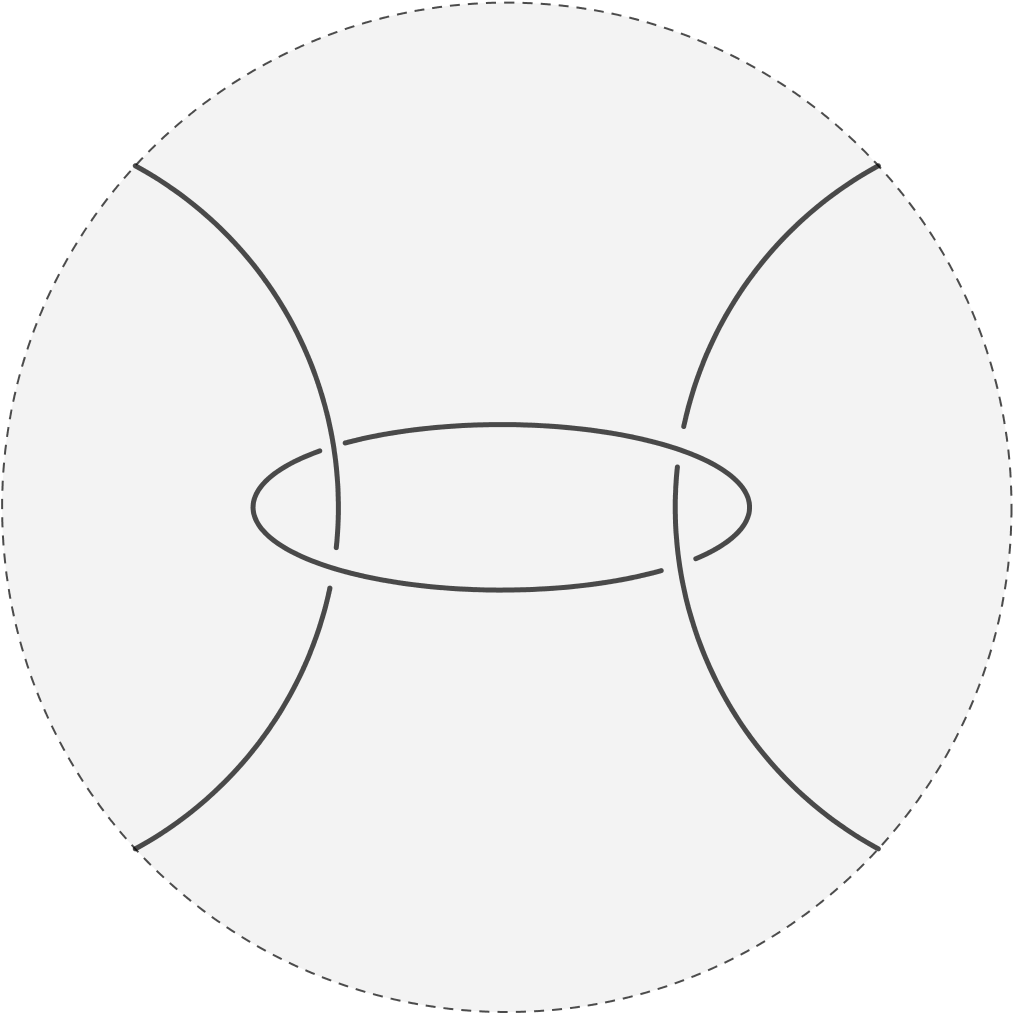}

(a)
\end{center}
\end{minipage}\hfill
$\longrightarrow$\hfill
\begin{minipage}{0.38\textwidth}
\begin{center}
\includegraphics[width=0.99\textwidth]{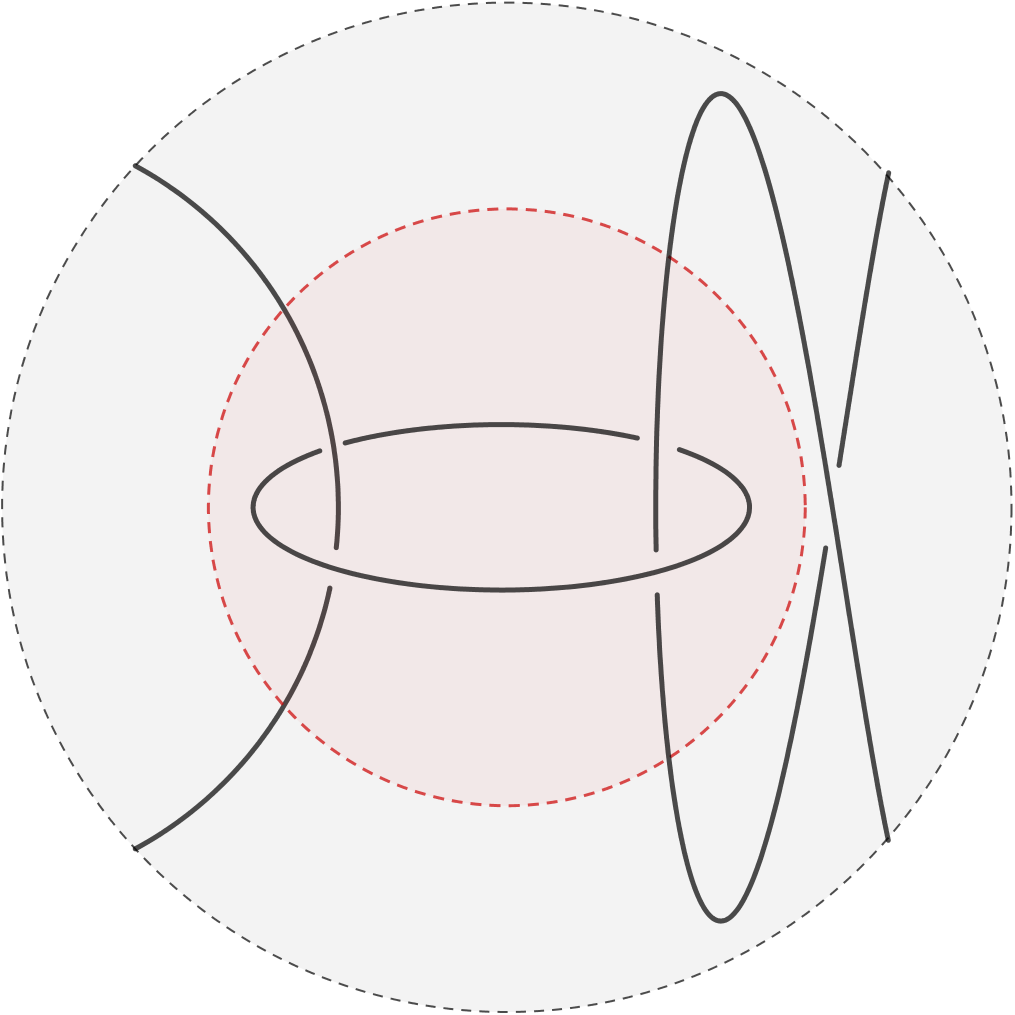}

(b)
\end{center}
\end{minipage}
\quad \quad
\caption{Using a Type I Reidemeister move to obtain a sub-ball
where the Chain Move Theorem applies.
\label{reidemeister}}
\end{figure}
}
\end{remark}

\begin{remark}{\em
Note that repeated application of the Chain Move Theorem allows us to create a
hyperbolic link complement
with an arbitrarily long chain of trivial components and with any
amount of twist.
Moreover, if the original exterior tangle is assumed not to be
rational, the subsequent exterior tangles
to which we apply the move cannot be rational either, so all resulting
link complements are hyperbolic.
In fact, even if the initial exterior tangle is rational, if our first
application of the move results
in a hyperbolic link complement, all repeated applications will also be hyperbolic.}
\end{remark}

\begin{proof}[Proof of Theorem~\ref{chain}]
Let $X = M\setminus L$ and, for $i=1,2$,
let $\G_i$ be the connected component of
$L$ containing the arc $\g_i$ (note that possibly $\G_1 = \G_2$).
First, we assume that $M$ is orientable.
We let $L'$ be the link formed by the replacement stated
in Theorem~\ref{chain}. As stated in the introduction
of this section, we will assume that
$(M \setminus \cB, L \setminus (\cB \cap L))$ is not a rational
tangle in a 3-ball; this special case can be described by
the following result, which is proved in the \hyperref[rationaltangles]{Appendix}.

\begin{lemma}\label{rational}
Let $L$ be a link in the 3-sphere such that the tangle $R = L\setminus \cB$
is a rational tangle and the tangle $L\cap \cB$ is the tangle $T_k$ appearing
in Figure~\ref{chainlemma1}(b), for some integer $k$. If $R$ is none of the
rational tangles $\infty$, $-k$, $-(k+1)$ or $-2 -k$, the link complement is hyperbolic.
\end{lemma}

Note that we do not include the rational tangle $\infty$ as a tangle to
exclude in the statement of Theorem~\ref{chain} since, in
the case of this tangle, the original link $L$ is not prime and
hence $X$ is not hyperbolic.
We prove Theorem~\ref{chain} when $M$ is orientable by showing that
the resulting link complement $Y = M\setminus L'$ does not admit
essential disks, spheres, tori or annuli. In order to do so, we first
prove the following.

\begin{claim}\label{Qincomp}
The four-punctured sphere
$\cQ = \cS \setminus L$ is incompressible and boundary-incompressible
in $X$ and also in $Y$.
\end{claim}
\begin{proof}
We prove that if $\cQ$ is compressible in $X$ or in $Y$, then
$(M \setminus \cB, L \setminus (\cB\cap L))$
is a rational tangle in a 3-ball. We first prove this property in $X$. Let $\g$
be a nontrivial simple closed curve in $\cQ$ and
assume that there is a compact disk $E\subset X$
with $\partial E = \g=E\cap \cQ$. Then, each of
the two disks $E_1$ and $E_2$ bounded by $\g$ in $\cS$
must contain
exactly two of the punctures of $\cQ$, otherwise
we could attach a one-punctured disk in $\cQ$ to $E$ to find
an essential disk in $X$, contradicting its hyperbolicity.

If $E$ were contained in $\cB$, then $E\cup E_1$ and
$E\cup E_2$ are two spheres in $\cB$, each punctured twice
by $L$. Since both punctures in each sphere cannot come from
distinct arcs in $L\cap \cB$, $E$ separates $\cB$
into two balls $B_1,B_2$, where $\g_1\subset B_1$
and $\g_2 \subset B_2$, and it then follows that $C$ cannot
link $\g_1$ and $\g_2$ simultaneously, a contradiction.

Next assume that
$E\cap{\rm int}(\cB ) = \es$.
Let, for $i=1,2$, $A_i = E\cup E_i\setminus L$;
then each $A_i$ is an annulus in $X$.
Since each $A_i$ is incompressible and $X$ is hyperbolic,
$A_i$ is boundary parallel. Therefore, the closure
of $A_i$ in $M$ bounds a closed ball
$B_i \subset M\setminus {\rm int}(\cB )$
with $\partial B_i = E_i\cup E$ and such that $B_i \cap L$ is
an unknotted arc in $B_i$. Hence, we can isotope $L \cap B_i$
through $B_i$ to the surface $\cS$.
Then, after the isotopy, $S = \partial N(\cB)$
is a sphere in $X$.
Since $X$ is hyperbolic, $S$ bounds a ball which is
disjoint from $\cB$, and this is a contradiction unless
$M = \sn3$.

If $M = \sn3$, then the fact $L \cap B_i$ can be isotoped
through $B_i$ to the surface $\cS$ implies
$L \setminus (\cB \cap L)$ can be isotoped to be
two disjoint embedded arcs on $\cS$. Hence,
$(M \setminus \cB, L \setminus (\cB \cap L))$ is
a rational tangle determined by $\gamma$, and up to isotopy, $E$ is
the only compression disk for $\cQ$ in $X$.

Note that the above argument implies that
$\cQ$ is also incompressible in $Y$, as we next explain.
If $E\subset Y$ was a compressing disk with
$\partial E = E\cap \cQ$,
then $E$ necessarily is contained in $\cB$,
otherwise $E \subset Y\setminus \cB$ and
$Y\setminus \cB = X \setminus \cB$
would give a compression disk for $\cQ$ in $X$.
Once again, $E\subset \cB$ gives that $E$
separates $\cB$ into two balls $B_1,\,B_2$
such that $\g_1 \subset B_1$ and $\g_2\subset B_2$.
Then, since $C_1$ links $\g_1$, $C_1 \subset B_1$. And since
$C_2$ links $\g_2$, $C_2 \subset B_2$.
But then $E$ separates $C_1$ from $C_2$ in $\cB$, a contradiction
to the fact they are linked in $\cB$.

To prove boundary-incompressibility of $\cQ$ in either $X$ or $Y$,
suppose $E$ is a boundary-compression
disk such that $\partial E = \alpha \cup \beta$ with
$\alpha = E \cap \cQ$. If $\alpha$ connects two distinct punctures
of $\cQ$ and $N(E)$ is a small neighborhood of $E$ in $M$, then
$\partial N(E) \setminus (\partial N(E) \cap \cB)$
is a compression disk for $\cQ$, a contradiction.

If both endpoints of $\alpha$ are at the same puncture, then,
since the interior of $\beta$ is disjoint
from $\cQ$, then $\beta$, together with an arc in $\cQ$,
bound a disk $\wt{E}$ in $\partial N(L)$. Then, $E\cup \wt{E}$
is a compression disk for $\cQ$, a contradiction.
\end{proof}

For the next arguments in the proof,
let $D_1,\,D_2 \subset Y\cap \cB$ denote two twice-punctured
disks bounded respectively
by $C_1,\,C_2 \subset L'$ and let
$\ol{D}_i$ denote the closure of $D_i$ in $M$; thus each $\ol{D}_i$
is a disk in $\cB$.
Then, we prove the following.

\begin{claim}\label{claimDi}
The twice punctured disks
$D_1$ and $D_2$ are incompressible and boundary-incompressible
in $Y$.
\end{claim}
\begin{proof}
Suppose there
were a disk $E\subset Y$, ${\rm int}(E)\cap D_i = \es$
with nontrivial boundary in $D_i$.
Since $\cQ$ is incompressible and we may assume
general position, any component in $E\cap \cQ$ is
a simple closed curve that is trivial both in $E$ and in $\cQ$.
Choose an innermost curve $\a\subset E\cap \cQ$
in the sense that the interior of the
disk $E'\subset E$ bounded by $\a$ does not intersect $\cQ$
and let $E''$ be the disk bounded by $\a$ in $\cQ$.
Then $E'\cup E''$ is a sphere that is either
in the hyperbolic manifold $X$ or in $Y\cap \cB$.
In either case, $E'\cup E''$ bounds a ball in $Y$
that can be used to isotope $E'$ to $E''$ and further to remove
$\a$ from the intersection $E\cap \cQ$. After repeating this disk
replacement argument a finite number of times,
we may assume that $E\subset \cB$.

Let $E'$ be the disk in $\ol{D_i}$ bounded
by $\partial E$. Then $E\cup E'$ is a sphere in $\cB$
which is punctured only once by at least one of the components
in $L'\cap \cB$, a contradiction that shows that $D_1$ and $D_2$
are incompressible in $Y$.


To finish the proof of Claim~\ref{claimDi},
we argue by contradiction and assume that,
for some $i\in \{1,2\}$, $D_i$ is boundary-compressible.
Then, there exists a disk $E$
such that $\partial E = \alpha \cup \beta$, where $\alpha$ and
$\beta$ are arcs, $\alpha = E \cap D_i$ and
$\beta \subset \partial N(L')$.
For simplicity, we assume that $i = 1$ and note that the case $i = 2$
can be treated analogously.
We also notice that Claim ~\ref{Qincomp} allows us to isotope $E$ in $Y$ to lie in $\cB$. 

Since $\be$ cannot join points in distinct components of
$\partial N(L')$, there are three cases to consider. First,
$\alpha$ could connect
$\partial N(C_1)$ to itself
and separate the two punctures on $D_1$. 
Let $E'$ be one of the two disks obtained by
removing $\alpha$ from $\ol{D_1}$, and we may choose
$E'$ in such a way that
$E'' = E \cup E'$ is a disk in $\cB$ with nontrivial boundary
in $\partial N(C_1)$.
But then $\partial E''$ must be a longitude of $\partial N(C_1)$
and any such disk would be
punctured twice by $L'\cap \cB$, a contradiction since
$E\cap L' = \es$ and $E'$ is only punctured once by $L'$.

The second possibility is that both endpoints of $\alpha$ are
in $\partial N(\g_1)$, with
$\alpha$ going around the puncture of $D_1$ that comes from $C_2$. 
Let $E'$ be the punctured disk cut by $\alpha$ from $D_1$ and let $E'' = E\cup E'$. So $E''$ is a once-punctured disk in $Y$. If $E''$ has trivial boundary in the boundary of $N(\gamma_1)$, then we have a sphere in $\cB$ punctured once by $C_2$, which cannot happen. Thus, $E''$ has nontrivial boundary in $\partial N(\gamma_1)$. Since $E\subset \cB$, the boundary of $E''$ is isotopic to a meridian curve on the boundary of $N(\gamma_1)$. Thus, after adding to $E''$ a meridianal disk in $N(\gamma_1)$, we have a sphere in $\cB$ punctured once by each of $\gamma_1$ and $C_2$, a contradiction, since any component entering a sphere in $\cB$ must also leave the sphere.

The last possibility is that $\alpha$ connects $\partial N(C_2)$ to itself 
and goes around the puncture $\gamma_1$ creates in $D_1$. Once again, let $E'$ be the once-punctured disk $\alpha$
cuts from $\D_1$ and let $E'' = E\cup E'$. Then, the closure of $E''$ is a disk in $\cB$ that is not punctured by $C_1$, and 
the fact that $\partial E''\subset \partial N(C_2)$ and the linking property between $C_1$ and $C_2$ implies that either
$\partial E''$ is trivial or it is a meridian in $\partial N(C_2)$ and the proof continues as previously,
finishing the proof of Claim~\ref{claimDi}.
\end{proof}

\begin{claim}
$Y$ does not admit any essential spheres or essential disks.
\end{claim}
\begin{proof}
We argue by contradiction and first suppose that
there is an essential sphere $S$ in $Y$.
If $S$ intersects $\cQ$,
then, by incompressibility of $\cQ$, we can
exchange disks on $S$ for disks on
$\cQ$ in order to obtain an
essential sphere $S'$ in $Y$ that does not intersect
$\cQ$. If $S'\subset Y\setminus \cB$, then $S'\subset X$,
which implies $S'$ is the boundary of a ball $B\subset X$.
In this case, $B$ must be disjoint from $\cB$, since $C\subset \cB$;
hence, $B\subset Y$ which contradicts that $S'$ is essential in $Y$.
Thus, we may assume that $S'$ is contained in $\cB$, and so it bounds a
sub-ball $B$ of $\cB$. If $B$ intersects $C_1\cup C_2$,
then, by the linking properties of these circles, $C_1\cup C_2$ must
be contained in $B$.
As $\g_1 $ links $C_1$ in $\cB$, we arrive at a contradiction
because the endpoints
of $\g_1$ lie outside of $B$. This contradiction implies that
$L'\cap \cB$ is disjoint from $B$,
which in turn implies that $B\subset Y$, contradicting that
$S'$ is essential in $Y$.

Suppose now that there is an essential disk $D$ with boundary
in $\partial N(L')$.
Then, there is a component $J$ of $L'$ such that
$\partial D\subset \partial N(J)$,
and we let $ S = \partial N(D \cup N(J))$. It then follows that
$S$ is an essential sphere,
as it splits $J$ from the other components of $L'$, contradicting
the nonexistence of such spheres.
\end{proof}

\begin{claim}\label{essann}
$Y$ does not admit essential annuli.
\end{claim}
\begin{proof}
Arguing by contradiction, assume there exists an essential
annulus $A$ in $M\setminus N(L')$.
Let $\a_1,\,\a_2$ denote the two
boundary components for $A$. Then, there are $J_1,\,J_2$
components of $L'$ such that $\a_1\subset \partial N(J_1)$
and $\a_2\subset \partial N(J_2)$.
After an isotopy of $A$ we will assume
without loss of generality that both $\a_1$ and $\a_2$ are
taut in the respective tori $\partial N(J_1)$,
$\partial N(J_2)$, in the sense that, in the product
structure generated by respective meridianal curves in
$\partial N(J_i)$,
each $\a_i$ is transverse to all meridians and also to all longitudes,
unless $\a_i$ is one of them.

We next rule out the various
possibilities for $A$, starting with the assumption
that $A$ does not intersect $D_1 \cup D_2$.

In this case, we may use the fact that
$\partial N(D_1\cup D_2)\setminus N(\g_1\cup\g_2)$
is isotopic to $\cQ$ to isotope $A$
in $M\setminus N(L')$ to lie outside of $\cB$.
Thus, $A$ is an annulus in $X$, and the fact that $X$ is hyperbolic
implies that $A$ is either compressible or boundary parallel
in $X$. If $A$ is compressible in $X$, then we may
use the fact that $\cQ$ is incompressible in $Y$
and a disk replacement argument to show that
$A$ is compressible in $Y$, a contradiction.

Next, we treat the case when
$A$ is boundary parallel in $M\setminus N(L)$.
In this case, $A$ defines a product region $W\subset M\setminus N(L)$
through which $A$ is parallel to a subannulus in $\partial N(L)$.
Since $C$ lies outside of $W$, separation properties
imply that $\cB$ is disjoint from $W$; hence,
$W\subset M\setminus N(L')$ from where it follows
that $A$ is boundary parallel in $Y$, a contradiction.

Now suppose that $A$ intersects $D_1 \cup D_2$ and
assume that $A$ has the fewest number of intersection
components in $A \cap (D_1 \cup D_2)$ for an essential annulus
in $Y$. Note that for $i=1,2$, the intersection
curves which may appear in $A\cap D_i$ are either simple closed
curves or arcs with endpoints in $\partial A$.

We next eliminate the possibility that
$A\cap D_i$ contains a simple closed curve.
Since $D_i$ is incompressible, by minimality of intersection curves,
any simple 
closed curve in the intersection $A \cap D_i$ is nontrivial in $A$.
Note that if $A\cap D_i$ contains a simple closed
curve that circles one puncture, we may take an innermost
such curve and
use the once-punctured disk on $D_i$ that it bounds to surger $A$
to obtain two annuli, each with fewer intersection curves and at
least one of them must be essential. So we may assume that all simple
closed curves in $A\cap D_i$ circle both punctures of $D_i$.
But then, the outermost of such intersection curves
bounds an annulus that again allows
us to surger $A$ to obtain an essential annulus with fewer
intersection curves.
Hence, all curves in $A\cap D_i$ are arcs
with endpoints in $\partial A$.

Next, we show that there are no arcs in $A \cap D_i$ that have
endpoints on the same boundary component of $A$. Assume that
$\alpha$ is such an arc and let $E_1$ be the disk
defined by $\a$ in $A$. We assume that $\a$ is
innermost in the sense that the interior of $E_1$
is disjoint from $D_i$.
Since Claim~\ref{claimDi} implies that $D_i$ is
boundary-incompressible,
it follows that $\a$ must cut a disk $E_2$ from $D_i$.
Then, $E = E_1 \cup E_2$ is a disk
with boundary $\partial E \subset \partial N(J)$.
Since $Y$ does not admit essential disks, it follows that
$\partial E$ is trivial in $\partial N(J)$, and
we may use the fact that all spheres in $Y$ bound
balls to isotope $A$ so that $E_1$ moves past $E_2$,
thus eliminating the
intersection curve $\a$ and
contradicting minimality of the number of intersection
components.

In particular, if $A$ intersects $D_i$, both $\a_1$ and
$\a_2$ must intersect $D_i$,
and none of the
intersection arcs on $A \cap D_i$ can cut a disk off $D_i$, as if they
did, $A$ would be boundary-compressible and hence boundary-parallel
since $Y$ is irreducible.
Note that because there is at
least one arc of intersection
of $A$ with a $D_i$, and such arc goes from $\a_1$ to $\a_2$, then
$\partial A\subset (\partial N(C_1)\cup \partial N(C_2)\cup
\partial N(\G_1) \cup \partial N(\G_2))$. Moreover,
since both $\a_1$ and $\a_2$ intersect $D_1\cup D_2$
and we assume minimality of intersection components
in $\partial A\cap (D_1\cup D_2)$,
no component of $\partial A$ can be a meridian in $\partial N(\G_1)$
or in $\partial N (\G_2)$, hence any closed curve in $A\cap \cQ$
must be trivial in $A$, and, consequently, trivial in $\cQ$.

We next consider the case that
$\partial A \subset \partial N(C_1) \cup \partial N(C_2)$.
Then by incompressibility of $\cQ$,
we can isotope $A$ to lie inside $\cB$.
Moreover, $C_1 \cup C_2$ is a Hopf link with complement
in the 3-sphere that is a
thickened torus $T \times [0,1]$. Thus,
$\cB \setminus (N(C_1) \cup N(C_2))$ is the complement of
a ball $B$ in $T\times [0,1]$, where we identify
$\partial N(C_1)$ with $T\times\{0\}$ and
$\partial N(C_2)$ with $T\times\{1\}$.

Assume that both boundary components of $A$ are on $\partial N(C_1)$.
Then, $A$ is an annulus
in $(T \times [0,1])\setminus B$ with both boundaries on $T \times \{0\}$.
In particular, in $T\times [0,1]$, $A$ is
boundary-parallel through a solid torus $V$
that $A$ cuts from $T \times [0,1]$.
Since $\partial V$ is a closed surface
in the interior of the three-ball $\cB$, it defines a unique compact
region disjoint from $\partial \cB = \partial B$,
from where it follows that $B$ must be disjoint from $V$.
But then both the arcs $\g_1$ and $\g_2$,
which have endpoints on $\partial \cB$, must also be disjoint from
$V$, meaning that $V\subset Y$,
and then $A$ is boundary-parallel in $Y$, a contradiction.
By symmetry, the same argument also proves
that $A$ cannot have both boundary components on $\partial N(C_2)$.

Next, suppose that one boundary of $A$ is on $\partial N(C_1)$ and
the other is on $\partial N(C_2)$.
Then again, $A$ can be seen as an annulus
in $(T\times [0,1]) \setminus B$, but now
its boundary is a pair of nontrivial parallel
curves on $T \times \{0\}$ and $T \times \{1\}$. These curves
are respectively realized as a $(p,q)$-curve\footnote{For given relatively
prime integers $p$ and $q$, a $(p,q)$-curve is a torus knot
that winds $p$ times around the meridian of the torus and
$q$ times around its longitude.}
on $\partial N(C_1)$
and a $(q,p)$-curve on $\partial N(C_2)$. But
there exist arcs $\wt{\g}_1$ and $\wt{\g}_2$ on $\cQ$
such that the closed curve $\g_1\cup \wt{\g}_1$
wraps meridianally around $C_1$ and
the closed curve $\g_2\cup \wt{\g}_2$
wraps meridianally around $C_2$,
where in $T\times [0,1]$, a meridian of $\partial N(C_2)$
corresponds to a longitude of $N(C_1)$. Hence, when we add $\g_1$
and $\g_2$ to $T\times [0,1] \setminus B$,
one wrapping meridianally around $T\times [0,1]$ and the other
wrapping longitudinally, at least one will puncture $A$,
a contradiction.

Thus, at least one boundary component of $A$, say $\a_1$, must
be on $\partial N(\G_i)$, for some $i\in\{1,2\}$.
As already explained, $\a_1$ is not a meridian on $\partial N(\G_i)$.

Next, assume that $\a_2$ is on $\partial N(C_1)$ or $\partial N(C_2)$.
Since $\cQ$ is incompressible and $Y$ is irreducible,
after performing a disk replacement argument, we may assume that
$A\cap \cQ$ is a family of pairwise disjoint arcs, each with both
endpoints in $\a_1$. Let $a$ be one of such arcs and assume that
$a$ cuts an innermost disk $D$ from $A$, in the sense that
$D\cap \cQ = a$. If $D\subset \cB$, then, if we let
$b=\partial D\setminus a$, it follows that
$b\subset (\partial N(\g_i))\cap \cB$ and
our assumptions on $\a_1$ being taut imply that $b$ joins two distinct
punctures of $\cQ$. But then it follows that $\partial D$
links $C_i$ on $\cB$, and $D$ must be
punctured by $C_i$, a contradiction.
Hence, it follows that
$D$ is to the outside of $\cB$.
Once again, our assumptions on $\a_1$ imply that
$a$ joins two distinct punctures of $\cQ$, from where it
follows that $D$ is a boundary-compression disk for $\cQ$,
which contradicts Claim~\ref{Qincomp}.

It remains to rule out the case where
$\a_1\cup\a_2\subset\partial N(\G_1)\cup \partial N(\G_2)$.
Let $a$ be an arc of intersection $A\cap (D_1\cup D_2)$.
Then our previous arguments give that $a$
joins $\a_1$ and $\a_2$ and that $a$ cannot cut a disk off $D_i$.
In particular, $a$ must necessarily intersect
the disk $D_j$ for $j\neq i$ and that creates another arc
$b\subset A\cap D_j$ which meets $a$ transversely at a point
$p$ and joins $\a_1$ and $\a_2$. In particular, $\G_1 = \G_2$.
The point $p$ separates both arcs $a$ and $b$, and that
defines a unique disk
$D\subset A$ with boundary given by one arc in $a$, one arc
in $b$ and one arc $c$ in $\a_1$.
Note that $D \cap D_i \subset a\cup b$, since any arc in
$A\cap D_i$ must join $\a_1$ and $\a_2$.
Let $E$ be a connected component of $D\setminus \cB$ that contains
a subarc of $c$ in its boundary.
Such component exists because the endpoints of $a$ and $b$ on $D$
are on distinct disks $D_1$, $D_2$
and hence $c$ cannot be contained in $\cB$. Once again,
the fact that $\a_1$ is taut
gives that $\partial E\cap \cQ$ is an
arc joining two distinct punctures of $\cQ$. But then, $E$ is
a boundary-compression disk for $\cQ$, a contradiction.

The cases treated above
rule out the possibility that $Y$ admits an essential annulus,
thereby proving Claim~\ref{essann}.
\end{proof}

\begin{claim}\label{clnotorus1}
$Y$ does not admit essential tori.
\end{claim}
\begin{proof}
We argue by contradiction and
suppose that $T\subset Y$ is a torus which is
incompressible and not boundary-parallel in $Y$.
First, suppose that $T$ does not intersect $D_1\cup D_2$.
Then we can isotope $T$ in $Y$ to assume that $T \cap \cB = \es$,
and then $T\subset X$. Since $X$ is hyperbolic,
either $T$ admits a compression disk in $X$ or
$T$ is boundary parallel in $X$.

First assume that $E\subset X$ is a compression disk for
$T\subset X\setminus \cB$.
Since $\cQ$ is incompressible in $X$, after disk replacements,
we may assume that
$E$ is disjoint from $\cQ$. In particular,
$E\subset X\setminus \cB\subset Y$, which is a contradiction.

Next, suppose that $T$ is parallel to the boundary of a neighborhood of
one of the components $J$ of $L$,
and let $W\subset X$ be the related proper product region with
boundary $T$. We claim that
$\cQ$ must be disjoint from $W$. Otherwise, $\cQ\subset W$ which
would imply that $\cB\setminus(\g_1\cup\g_2\cup C)\subset W$;
this is a contradiction because
$W$ has only one end corresponding to a single component of $L$.
Since $\cQ$ separates $X$ and is disjoint from $W$, then $W\subset Y$,
which means $T$ is boundary parallel in $Y$.
This proves that any essential torus in $Y$ must intersect
$D_1\cup D_2$.

Let $T\subset Y$ be an essential torus that
intersects $D_i$, for some $i\in\{1,2\}$.
Next, we prove that $Y$ must
contain an essential annulus,
which contradicts Claim~\ref{essann}.
After possibly replacing disks in $T$ by disks in the
incompressible surface $D_i$, we may assume that any component
in $T\cap D_i$ is homotopically nontrivial in $D_i$;
let $\g\subset T\cap D_i$ be one such components.
First assume that $\g$ encircles a single puncture in $D_i$ and
choose it to be
an innermost such curve in $T\cap D_i$.
Using the once-punctured disk bounded by $\g$ in $D_i$ to surger
$T$, we obtain an essential annulus in $Y$, as claimed.
Next, assume that $\g$ encircles both
punctures of $D_i$ and that it is an outermost such
curve on $T\cap D_i$.
In this case, we may use the outer annulus on $D_i$ to surger
$T$ in order
to obtain an essential annulus in $Y$, thereby proving
Claim~\ref{clnotorus1}.
\end{proof}

Having proved that there are no essential disks, spheres, tori
or annuli in $Y$, then $Y$ satisfies
Thurston's conditions for hyperbolicity, proving
Theorem~\ref{chain} when $M$ is orientable.

The case when $M$ is nonorientable can be proved using the
orientable case as we next explain.
Suppose that $M$ is nonorientable and that $L$, $L'$ and $\cB$ are
as stated. Let $\Pi\colon \wh{M}\to M$ be the oriented
two-sheeted covering of $M$ and let $\wh{L} = \Pi^{-1}(L)$
and $\cB_1,\,\cB_2$ be the two connected components of $\Pi^{-1}(\cB)$.
Then, $\wh{L}$ is a hyperbolic link in $\wh{M}$ and
$\wh{L}\setminus \cB_1$ is not a rational tangle
in a 3-ball, since $\wh{L}\cap \cB_2$ is diffeomorphic to
$L\cap \cB$.
Then, we may use the chain move to modify $\wh{L}$ in $\cB_1$,
replacing $\wh{L}\cap \cB_1$ by a tangle diffeomorphic to
$L'\cap \cB$, which
creates a hyperbolic link $\wh{L}'$ in $\wh{M}$.
Then, since $\wh{L}'\cap \cB_2 = \wh{L}\cap \cB_2$
and $\wh{M}\setminus \wh{L}'$ is hyperbolic,
we can use the chain move in $\cB_2$
to replace $\wh{L}'\cap \cB_2$ by a tangle diffeomorphic to
$L'\cap \cB$ and create another hyperbolic link $\wh{L}''$ in $\wh{M}$.
Since we may do this second replacement
in an equivariant manner with respect to the nontrivial covering
transformation $\sigma$ defined by $\Pi$, the restriction of $\Pi$
to the hyperbolic manifold $\wh{M}\setminus \wh{L}''$ is the two-sheeted
covering space of $M\setminus L'$. Since $\sigma$ is an
order-two diffeomorphism of $\wh{M}\setminus\wh{L}''$, the
Mostow-Prasad rigidity theorem implies that we may consider $\sigma$
to be an isometry of the hyperbolic metric of $\wh{M}\setminus\wh{L}''$.
Hence, the hyperbolic metric of $\wh{M}\setminus \wh{L}''$
descends to $M\setminus L'$ via $\Pi$,
which finishes the proof of Theorem~\ref{chain}.
\end{proof}

\section{The Switch Move Theorem}\label{secSwitch}

\begin{theorem}[{Switch Move Theorem}]\label{switch}
Let $L$ be a link in a 3-manifold $M$
such that $M\setminus L$ admits a complete hyperbolic metric of
finite volume. Let $\a\subset M$ be a
compact arc which intersects $L$ transversely in its two distinct
endpoints, and such that int$(\a)$ is a complete, properly
embedded geodesic of $M\setminus L$. Let
$\cB$ be a closed ball in $M$ containing $\a$ in its interior and
such that
$\cB\cap L$ is composed of two arcs in $L$, as
in Figure~\ref{switchbefore}.
Let $L'$ be the resulting link in $M$ obtained by replacing
$L \cap \cB$ by the components as appearing in
Figure~\ref{figaugmented}~(b).
Then $M\setminus L'$ admits a complete hyperbolic metric of
finite volume.
\end{theorem}

\begin{proof}
We begin the proof by setting the notation.
Let $G$ and $G'$ be the connected components of
$L$ containing the arcs $g$ and $g'$ respectively,
as in Figure~\ref{figaugmented}(a).
Note that it can be the case $G = G'$.
Let $L'$ be the link formed by replacing $g\cup g'$ in $\cB$
by $\gamma_1\cup \gamma_2 \cup C$.
For $i=1,2$, let $\G_i$ be the component of $L'$ containing
$\g_i$. Note that possibly $\G_1 = \G_2$ and let $\G = \G_1 \cup \G_2$.

We split the proof into two cases, depending on whether or not
$(M \setminus \cB, L \setminus (\cB \cap L))$
is a rational tangle in a 3-ball.

\begin{claim}If $(M \setminus \cB, L \setminus (\cB \cap L))$ is a
rational tangle in a 3-ball, then $M \setminus L'$ is hyperbolic.
\end{claim}

\begin{proof} A rational tangle in a 3-ball always has a projection
that is alternating (see for instance~\cite{GK}).
Then, $L$ is a rational, alternating link in $S^3$ that is
prime, since $M\setminus L$ is hyperbolic.
By Corollary~2 of~\cite{mena},
a rational, alternating link in $S^3$ that is
prime is hyperbolic
if and only if it is nontrivial and not a 2-braid.
After forming $L'$,
we consider the link $L''$ obtained from $L'$ by doing a half-twist on
the twice-punctured disk bounded by $C$
to add a crossing so that $L'' \setminus C$ has an alternating
projection, as in Figure~\ref{rationaltangle}.
Then, $L''$ is in an augmented alternating link projection obtained
from a prime, non-split reduced alternating projection.
If $L'' \setminus C$ is neither trivial nor a 2-braid, $L''$ is
hyperbolic by~\cite{A1}. However, if $L'' \setminus C$ is trivial,
then $L$ is a 2-braid and hence it does not satisfy the hypothesis that
$M \setminus L$ is hyperbolic.
And if $L'' \setminus C$ is a 2-braid, then $L$ is a trivial knot,
again not satisfying the same hypothesis.
So $L''$ is a hyperbolic link in $S^3$. But from Theorem~4.1
of~\cite{ad1}, $L''$ is hyperbolic if and only if $L'$ is hyperbolic.
\end{proof}

\begin{figure}[htpb]
\begin{center}
\includegraphics[width=1.0\textwidth]{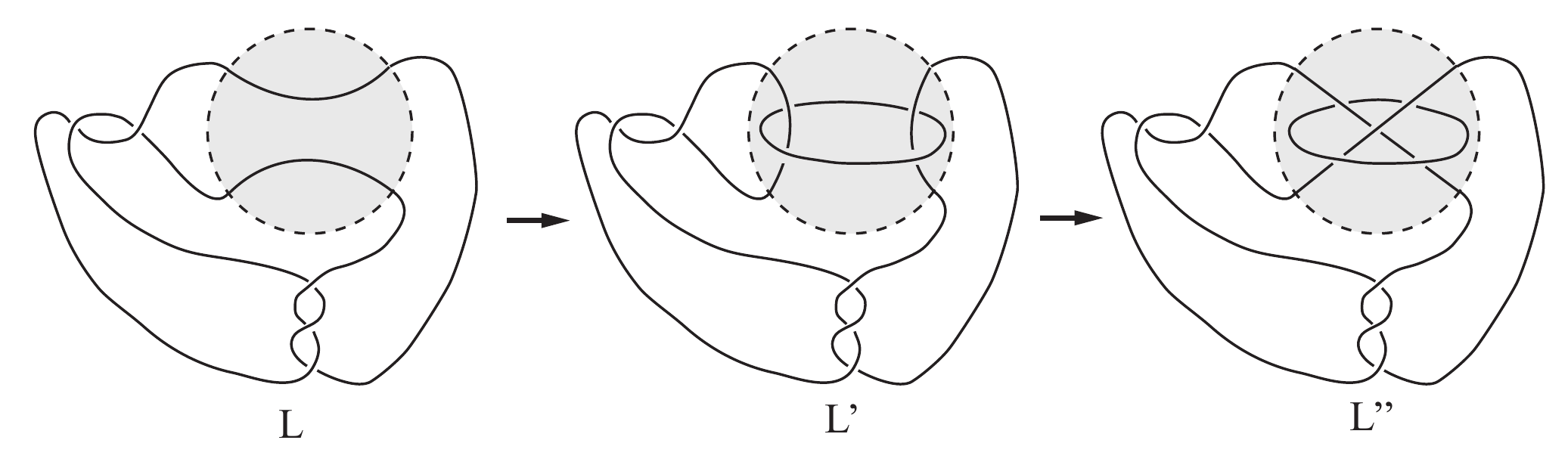}
\caption{Creating $L$, $L'$ and $L''$.}
\label{rationaltangle}
\end{center}
\end{figure}

\begin{remark}\label{remarkgeodesic}{\rm
If $M \setminus L$ is hyperbolic and
$(M \setminus \cB, L \setminus (\cB \cap L))$ is a rational tangle in
a 3-ball, then $L$ is either a rational link or a rational knot in
$S^3$ which is hyperbolic. In this case, there is always an arc
$\alpha$ as depicted in Figure~\ref{switchbefore} which is isotopic to
a geodesic and hence the switch move can be applied. This follows
because $\alpha$ is part of the fixed point set of an involution of the
complement, which is realized by an isometry, and fixed-point sets of
isometries must be geodesics (see~\cite{AR} for the details).}
\end{remark}

From now on, we assume that
$(M \setminus \cB, L \setminus (\cB \cap L))$ is not a rational
tangle in a 3-ball.
As in the proof of the Chain Move Theorem (Theorem~\ref{chain}), we
first assume that $M$ is orientable. We also
let $X = M\setminus L$ and $Y = M\setminus L'$
and we will prove that $Y$ is hyperbolic by showing
that there are no essential disks,
spheres, tori or annuli in $Y$.
Once again, we let $\cS = \partial \cB$,
$\cQ = \cS\setminus L= \cS\setminus L'$
and notice that the same arguments used to prove Claim~\ref{Qincomp}
can be used to prove that $\cQ$
is incompressible and boundary-incompressible in $Y$;
the details are left to the reader.

Let $D$ be the interior of a
twice-punctured disk in $\cB\setminus L'$ bounded by $C$ and
let $\ol{D}$ be its closure in $M$.

\begin{claim} $D$ is incompressible and boundary-incompressible
in $Y$.
\end{claim}

\begin{proof}
Using the facts that $\cQ$ is incompressible
in $Y$, that $X$ is hyperbolic
and $Y\setminus \cB = X\setminus \cB$,
we may use a disk replacement argument to assume that any
compression disk for $D$ is contained in $\cB\setminus L'$.
Arguing by contradiction, assume that $E\subset \cB\setminus L'$
is a disk with $\partial E = E\cap D$, and that $\partial E$ is
nontrivial in $D$. Let $E_1\subset \ol{D}$ be the subdisk bounded
by $\partial E$ in $\ol{D}$. Let
$S = E_1 \cup E$. Then, $S$ is a two-sphere in the ball $\cB$ which
is punctured only once by
at least one of the arcs $\g_1$ or $\g_2$, which is impossible.

Next, we prove that $D$ is boundary-incompressible.
Arguing by contradiction, let $E$ be a boundary-compression
disk for $D$. By incompressibility and boundary-incompressibility of $\cQ$,
we can assume that $E\subset \cB$.

First, consider the case when $E\cap D$ is an arc $\be$ with
endpoints in $\partial N(C)$ separating the punctures in $D$. Let $D'$
be the closure in $\ol{D}$ of
one of the punctured disks that results from that separation. Then
$E' = E \cup D'$ is a disk in $\cB \setminus N(C)$ with boundary in
$\partial N(C)$ and
$\partial E'$ must either be trivial or
longitudinal on $\partial N(C)$. If $\partial E'$
is trivial, we can form a sphere
by adding to $E'$ the disk on $\partial N(C)$ also
bounded by $\partial E'$, thereby creating a once-punctured sphere in
$\cB\setminus L'$, a contradiction. If, on the other hand,
$\partial E'$ is a longitude on $\partial N(C)$, then the fact $E'$ is
only punctured once by $L'$ is a contradiction to its
construction in $\cB$.

Suppose now that $E\cap D$ is an arc $\be$ from one puncture of
$D$ to the other. Then, $\partial E\setminus \beta$ is an arc
that cannot be contained in $\cB$, a contradiction.

Suppose now that $E\cap D$ is an arc $\be$ on $D$ that begins
and ends at the same puncture and goes around the second puncture.
Then $\partial E$ cannot link $C$ since $E$ is unpunctured.
So $\partial E \cap \partial N(\G)$ can be isotoped on
$\partial N(\G)$ into $D$. But then, $E$ becomes a compression disk
for $D$, a contradiction.
\end{proof}

Using that both $D$ and $\cQ$ are incompressible and boundary
incompressible, we next proceed with the proof of Theorem~\ref{switch}.

\begin{claim}
$Y$ does not admit essential spheres or essential disks.
\end{claim}
\begin{proof}
We first show that there are no essential spheres in $Y$.
Suppose that $S\subset Y$ is a sphere
and first assume that $S\cap \cB = \es$.
Then, $S \subset X$,
and, since there are no essential spheres in
$X$, it follows that $S$ bounds a ball
$B\subset X$. Since $L$ intersects $\cB$,
this gives that $B\cap \cB = \es$, hence
$B\subset Y$, proving that $S$ is not essential in $Y$.

Next, we treat the case where
$S$ intersects $\cB$.
We can take $S$ to have the least number of
intersection curves in $S \cap \cQ$ over all essential spheres.
If $S$ were contained in $\cB$, it bounds a ball in $\cB$
which is also a ball in $Y = M\setminus L'$, since $S\cap L' = \es$.
Next, we assume that $S\cap \cQ \neq \es$.
Then, there
exists a disk $E\subset S$ with $\partial E = E\cap \cQ$.
After a standard disk replacement argument using that
$\cQ$ is incompressible and that there are no essential spheres
that do not intersect $\cQ$,
we isotope $S$ to lower the number of components
in $S\cap \cQ$, which proves
that there are no essential spheres in $Y$.

To prove that there are no essential disks in $Y$,
we argue by contradiction and assume that $E$ is such a disk with
boundary on a regular neighborhood of a component $J$ of $L'$.
Then, $S = \partial N(E \cup N(J))$ is an essential sphere in $Y$, as
it splits $J$ from the other components of $L'$, a contradiction.
\end{proof}

\begin{claim}\label{cl:annuli}
There are no essential annuli in $M\setminus N(L')$.
\end{claim}
\begin{proof}[Proof of Claim~\ref{cl:annuli}]
Suppose that $A$ is an essential annulus in $M\setminus N(L')$.
Consider first the case $A \cap \cB = \es$.
Then, $A\subset M\setminus N(L)$ and it must
either compress or be boundary-parallel in $M\setminus N(L)$.
First, let us assume that
$E\subset M\setminus N(L)$ is a compression
disk to $A$ with boundary $\beta$.
Then $\beta$ separates $A$ into two sub annuli $A_1$ and
$A_2$, and $E\cup A_1$ and $E\cup A_2$ give rise to two essential disks
in $X$, which contradicts its hyperbolicity.

Hence, we may assume that $A$
is boundary-parallel in $M \setminus N(L)$.
Then, there is a component $J$ of $L$ and an annulus
$A'\subset \partial N(J)$
such that $\partial A'= \partial A$ and
$A\cup A'$ bounds a solid torus $W$ in $M \setminus N(L)$,
through which $A$ is parallel to $A'$.
If $\cB\cap W = \es$, then $A$ is boundary parallel in $Y$,
a contradiction. Hence, we may assume that $\cB\cap W \neq \es$.
Since $A\cap \cB = \es$
and $L\cap W =\es$, then $A'$ must intersect $\cB$ and
$J$ must be either $G$ or $G'$, which could be the same component.
Suppose first that $G$ and $G'$ are distinct. Then if $\lambda$ is an
arc in $\cB\setminus N(L)$
with an endpoint in $\partial N(G)$ and the other in
$\partial N(G')$, at least one
endpoint of $\lambda$ is not in $W$. Since
${\rm int}(\lambda)$ cannot intersect $A$, it follows that
$G'\subset W$ a contradiction.

Suppose now $G$ and $G'$ are the same component $J$. Since
$A \cap \cB=\es$, $\partial A$ is a pair of meridians on
$\partial N(J)$. Then, there is a ball $B'$ in $N(J)$ bounded by $A'$
and two meridianal disks in $N(J)\setminus \cB$ bounded by
$\partial A$. Then $W' = W \cup B'$ is a ball in $M$, and $J\cap W'$
is an unknotted properly embedded arc within it.
Since $\cB \cap W \neq \es$, then $\cB \cap W'\neq \es$. But then,
the fact that $\partial W'\cap \cB=\es$ implies that
$\cB\subset W'$. Hence $\a$ can
be homotoped into $\partial N(J)$, contradicting the fact it is
a geodesic with endpoints on $L$. So $A \cap \cB \neq \es$.

Next, we treat the case $A\subset \cB$.
Let $\a_1,\,\a_2$ denote the two components of $\partial A$.
First, we assume that $\a_1\subset \partial N(\G)$ and
$\a_2\subset \partial N(C)$. Since $A\subset \cB\setminus N(L')$,
$\a_1$ is either a meridian of $\partial N(\g_1)$
or a meridian of $\partial N(\g_2)$, and the symmetry
between $\g_1$ and $\g_2$ allows us to assume
$\a_1\subset \partial N(\g_1)$. Take
a meridianal disk $E_1$ in $N(\g_1)\cap \cB$ with $\partial E_1 = \a_1$.
Then, $E = A\cup E_1$ is a disk in $\cB\setminus N(C)$
with $\partial E = \a_2\subset \partial N(C)$.
Hence, $\a_2$ is a longitude of $\partial N(C)$.
In particular, $\a_2$ links $\g_2$ in $\cB$, and hence $\g_2$ must
puncture $E$, which is a contradiction.
This contradiction shows that if
$A\subset \cB\setminus N(L')$ is
an essential annulus, then $\a_1$ and $\a_2$ are either both
parallel curves on $\partial N(C)$ or both meridians
on $\partial N(\G)$.

Assume that $A$ is an essential annulus in $M\setminus N(L')$
such that $A\subset \cB$ and $\a_1$ and $\a_2$ are meridians
on $\partial N(\G)$.
Let $E_1,\,E_2$ be two meridianal disks in $N(\G)$ with
respective boundaries $\a_1,\,\a_2$. Then, $A \cup E_1\cup E_2$
is a sphere in $\cB$ that bounds a ball $B\subset \cB$, which
is either punctured once by each $\g_1$ and $\g_2$, which is
not possible, or twice by one of them, say $\g_1$.
Since $A$ is not boundary parallel, then
$C\subset B$. However, since $C$ links both $\g_1$ and $\g_2$,
$\g_2$ must be contained in $B$, which is a contradiction.

Still assuming that $A\subset \cB$,
it remains to obtain a contradiction
when both $\a_1$ and $\a_2$
are $(p,q)$-curves on $\partial N(C)$.
In this case, $\cB\setminus(N(\g_1)\cup N(C))$
is diffeomorphic to $T\times [0,1]$, where $T = \sn1\times \sn1$
is a torus, and we identify
$\partial N(C)$ with $T \times \{0\}$.
Since any annulus in $T\times [0,1]$ with
boundary in $T\times\{0\}$ is parallel
to an annulus in $T\times\{0\}$,
it follows that $A$ is parallel to
an annulus $A'\subset \partial N(C)$ with
$\partial A'=\a_1\cup \a_2$,
in the sense that there is a
solid torus region
$W\subset T\times [0,1]$
with $\partial W = A\cup A'$.
Since $N(\g_2) \subset T\times [0,1]$ and does not
intersect $\partial W$,
the fact that the endpoints of $\g_2$ lie in $T\times\{1\}$
implies that $N(\g_2)$ is disjoint from $W$. Therefore,
$A$ is boundary parallel in $\cB\setminus N(L')$,
contradicting the assumption that $A$ was essential.

From now on, we will assume that $A$ intersects $\cQ$.
We will also assume that $A$ minimizes the
number of intersection curves of an essential annulus
of $M\setminus N(L')$ with $\cQ$. In particular,
since $\cQ$ is incompressible, the connected
components of $A\setminus \cQ$ are either annuli
or disks whose boundary intersect $\partial A$.

Suppose first that there is an intersection arc $a$ in $A \cap \cQ$
that cuts a disk $E$ from $A$.
Then both endpoints of $a$ are on the same boundary component of $A$
and $E\cap \cQ\subset \partial E$.
Because $\cQ$ is boundary-incompressible, it must be the case that
$a$ cuts a disk $E_1$ from $\cQ$. Then $E_2 = E \cup E_1$ is a disk
properly embedded in $M \setminus N(L')$. Since there are no
essential disks in $M\setminus N(L')$, then
$\partial E_2$ bounds a disk $E_3$ in $\partial N(L')$. Then,
$E_2 \cup E_3$ is a sphere that bounds a ball in $M\setminus N(L')$,
through which $E$ can be isotoped to $E_1$, and just beyond to
eliminate $a$ from $A\cap \cQ$, contradicting that we assumed a
minimal number of intersection components.

Thus, we now know that there are only two possibilities for the
intersection curves in $A \cap \cQ$. Either they are all parallel
nontrivial closed curves on $A$ or
they are all arcs with endpoints on distinct boundary components
of $A$.

If there are no arcs in $A\cap \cQ$, then
$\partial A\cap\cQ = \es$.
Since $A$ and $\cQ$ are incompressible in $M\setminus L'$,
the minimality condition on the number of curves in
$A\cap \cQ$ implies that any curve in $A\cap \cQ$ is nontrivial on
both $A$ and on $\cQ$.

Next, we prove that any curve in $A\cap \cQ$ must
encircle two of the punctures of $\cQ$. Arguing by contradiction,
assume that $a$ is a simple closed curve in $A\cap \cQ$
and assume that $a$ bounds a once-punctured disk $E$ in $\cQ$.
Without loss of generality, we may assume that $E$
is innermost in the sense that $E\cap A = a$.
Using $E$ to surger $A$, we obtain
two annuli in $M\setminus L'$, where at least one
is still essential, and, after a small isotopy, with a
lesser number of intersection components with $\cQ$,
which is a contradiction. Thus, any
curve in $A\cap \cQ$ encircles two of the punctures of $\cQ$
and all intersection curves must be parallel on $\cQ$, separating
one pair of punctures from the other pair.

Still under the assumption that $\partial A \cap \cQ = \es$ and
$A\cap \cQ\neq \es$, we next rule out the case
where at least one boundary component of $A$, say $\a_1$, lies in $\partial N(C)$.
In this case, let $A_1$ be the
connected component of $A\cap \cB$ containing $\a_1$ and let
$a = \partial A_1\setminus \a_1$ denote the other boundary component
of the annulus $A_1$.
Let $E$ be one of the two disks defined by $a$ in $\cS$.
Then, $A_1\cup E$ is a disk in $\cB\setminus N(C)$ which has nontrivial
boundary in $\partial N(C)$; hence,
$\a_1$ is a longitude. After an isotopy on $A_1$,
we may assume that $\a_1 \cap D = \es$, and thus
$\partial A_1 \cap D = \es$. Since $D$ is incompressible,
we may isotope $A_1$ in $\cB\setminus L'$ to assume that $A_1\cap D$
does not contain any trivial curves. Moreover, if $\be\subset A_1\cap D$
is a nontrivial simple closed curve both in $D$ and in $A_1$,
then $\be$ cannot encircle one puncture in $D$, since
this would generate a thrice-punctured sphere in $\cB\setminus L'$,
a contradiction. Hence, any curve in $A_1\cap D$ encircles both
punctures of $D$; this gives rise to solid tori regions in
$\cB\setminus N(L')$ that can be used to further isotope $A_1$ in
$\cB\setminus L'$ to assume that $A_1\cap D =\es$.
In particular,
after capping $\a_1$ with a longitudinal disk in
$\cB\setminus (N(C)\cup A_1)$,
it follows that $a$ is the boundary of a disk in $\cB\setminus L$.

Since any other curve in $A\cap \cQ$ must be parallel
to $a$, $A\cap \cQ$ is a family
$\{a_1,\,a_2,\,\ldots,\,a_n\}$ of pairwise disjoint
simple closed curves, all parallel to each other
both in $\cQ$ and in $A$. In particular, for each
$i\in\{1,\,2,\,\ldots,\,n\}$, $a_i$ generates $\pi_1(A)$ and
bounds a disk $E_i\subset \cB\setminus L$, punctured once by the
arc $\a$. Note that $n\geq 2$, since otherwise $\a_2\subset \partial N(J)$,
where $J$ is a component of $L$ and then capping $A$ with a
disk in $\cB\setminus L$ bounded by $\a_1$ would yield an essential
disk in $X$.
This implies that there exists a subannulus
$A_2 \subset A\setminus \cB$ with boundary $\partial A_2 \subset \cQ$.
Let us assume that $\partial A_2 = a_1\cup a_2$.
Then (after possibly isotoping the disks $E_1,\,E_2$ in
$\cB\setminus L$ so they become disjoint) $S = A_2\cup E_1\cup E_2$
is a sphere in $X$, which bounds a ball $B\subset X$. Let
$V = B \setminus \cB$, then $V$ is a solid torus in
$X\setminus \cB = Y\setminus \cB$ and we may use $V$
to isotope $A$ in $Y$ to reduce the number of intersection components
in $A\cap \cQ$, a contradiction.

In this point on the proof, assuming that $\partial A \cap \cQ = \es$
and that $A\cap \cQ \neq \es$, it remains to rule out the case where
no boundary component of $A$
is on $\partial N(C)$. Then, $\partial A \cap \cB =\es$, since
otherwise a boundary component of $A$ would be a meridian in
$\partial N(\G)$ and we could isotope $A$ to reduce the number of
intersection components in $A\cap \cQ$.
Next, we show that, after an isotopy, $A\cap D = \es$.
Indeed, since $D$ and $A$
are both incompressible,
after a disk replacement argument we may assume that
any curve in $A\cap D$ is a simple closed curve that
generates $\pi_1(A)$ and either encircles one or two of the punctures
of $D$. If there is a curve $a\subset A\cap D$, we may assume that
either $a$ encircles one puncture of $D$ and is innermost
or that $a$ encircles the two punctures of $D$ and is outermost.
In either case, we can surger $A$ to obtain two annuli in $Y$,
where at least one is still essential in $Y$ and with less intersection
components with $\cQ$, a contradiction that proves that $A\cap D = \es$.
Next, the identification
$\cB\setminus N(D\cup L') \equiv \cQ\times [0,1]$ gives us that
we may isotope $A$ in $Y$ to make $A$ disjoint from $\cB$. Since
we already showed that there are no essential annuli in $Y$ disjoint
from $\cB$, this is a contradiction.

This shows that if $A$ is an essential annulus in $Y$,
then all intersection curves in $A \cap \cQ$ are arcs, the endpoints of
each of which lie on distinct boundary components
of $A$. Such arcs cut $A$ into
a collection of disks. Because $\cS$ separates $M$, there must be an
even number of such arcs and hence such disks, and the disks must
alternate between lying inside and outside $\cB$.

There are no such arcs that cut a disk from $\cQ$. Indeed, if there were
such a disk, by choosing an innermost one, we could surger $A$ along
this disk to obtain a disk $\Delta$ with boundary in $\partial N(L')$.
Since there are no essential disks in $M \setminus N(L')$,
$\partial \Delta$ must bound a disk
$\Delta'$ on $\partial N(L')$. Then
$\Delta \cup \Delta'$ is a sphere bounding a ball in
$M\setminus N(L')$. Hence, we can isotope $\Delta$ to $\Delta'$
through the ball, and hence isotope $A$ to an annulus in
$\partial N(L')$, contradicting the fact that $A$ is not
boundary parallel in $M \setminus N(L')$.

Let $E$ be a connected component of $A \cap \cB$, which necessarily is
a disk. Next, we show that there are two possibilities for $E$ up to
isotopy and switching the roles of $\g_1$ and $\g_2$. These two
possibilities are depicted in Figure~\ref{disksII}.

\begin{figure}[htpb]
\begin{center}
\includegraphics[width=0.9\textwidth]{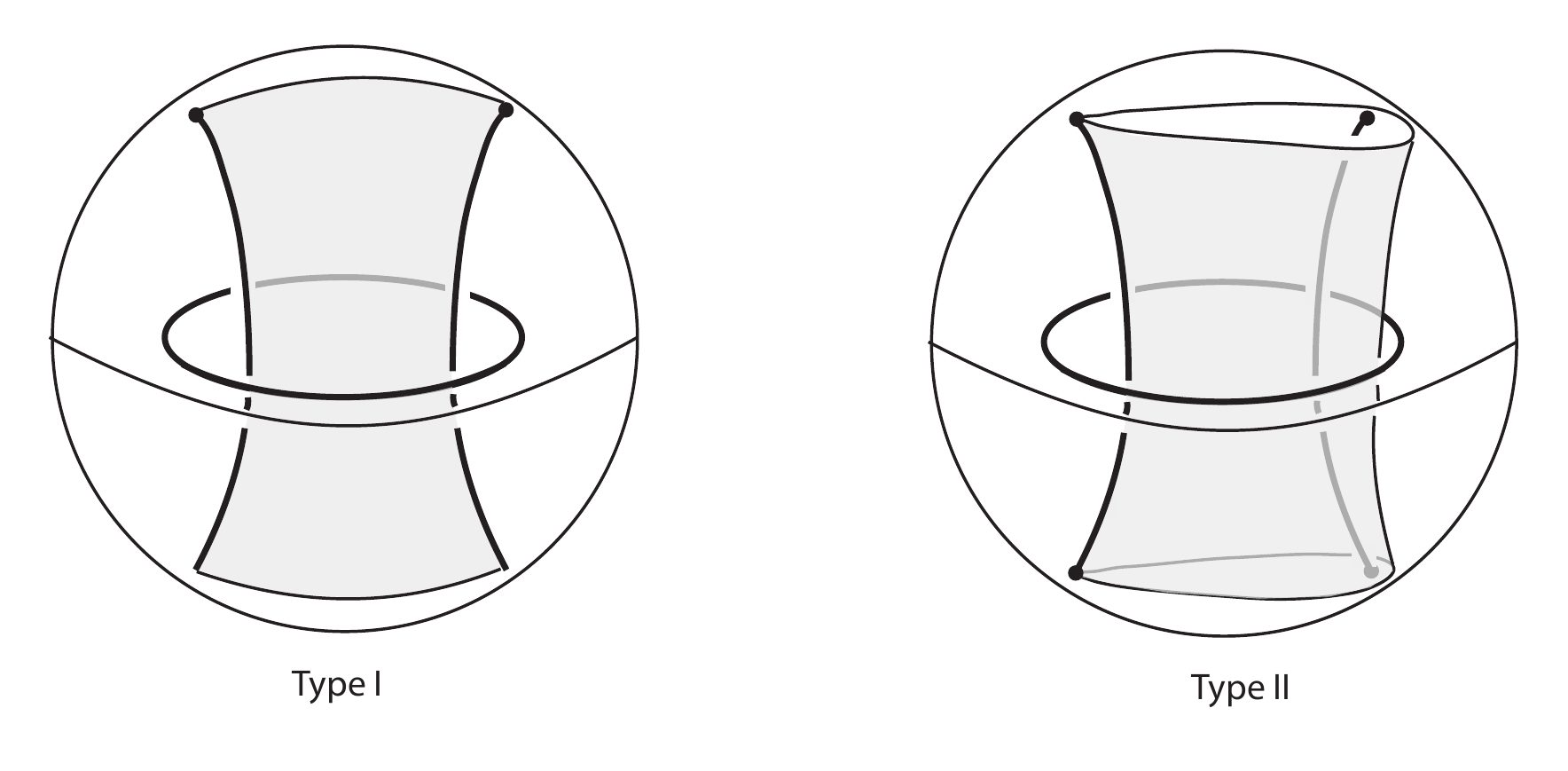}
\caption{Possibilities for $E$, a connected component of
$A \cap \cB$ when all intersections of $ A\cap \cQ$ are arcs.}
\label{disksII}
\end{center}
\end{figure}

Let $\partial E = \be_1 \cup \mu_1 \cup \be_2 \cup \mu_2$ where
$\be_1$ and $\be_2$ lie in $\cQ$ and $\mu_1$ and $\mu_2$ lie
in $\partial N(L')$. Note that each of $\mu_1$ and $\mu_2$ must
begin and end at distinct components of $\partial N(\G)\cap \cQ$,
since otherwise we could lower the number of intersection curves
of $A$ with $\cQ$.

For the arguments that follow, we set coordinates and consider
$\cB = \{(x,y,z)\in \rth\mid x^2+y^2+z^2\leq1\}$,
$D$ a horizontal disk in $\{z = 0\}$ and
the two arcs $\g_1,\,\g_2$ parallel to the $z$-axis.
Let $A'$ be the annular connected component of
$(\cB\setminus N(C))\cap \{z=0\}$. Then, $A'$ is annulus with
one boundary component in $\cQ$ and the other boundary component a
longitude on $\partial N(C)$.

We assume that we have isotoped $E$ in $\cB\setminus N(L')$ to
minimize the number of intersection curves in $E\cap A'$, and
next we prove that $E\cap A'= \es$.
First, we claim that $E\cap A'$ does not contain any arc. Indeed,
if there were an arc $\phi\subset E\cap A'$, 
since $\partial E\cap \partial N(C) = \es$, $\phi$ would cut a
disk $H_1$ from $A'$ and a disk $H_2$ from $E$.
Let $H_3 = H_1 \cup H_2$.
If $\phi$ has both endpoints in the same $\be_i$, then
$\partial H_3 \subset \cQ$, which, by incompressibility of $\cQ$,
implies that $\partial H_3$ is a trivial curve bounding a disk
$H_4 \subset \cQ$. Then $H_3 \cup H_4$ is a sphere bounding a ball,
through which we can isotope $H_1$ through $H_2$, and lower the number
of intersection curves in $E \cap A'$, a contradiction.

If $\phi$ has one endpoint in $\be_1$ and the other in $\be_2$, then
$\partial H_3$ consists of one arc in $\cQ$ and one 
taut arc on $\partial N(\g_i)$.
Then, we can use $H_3$ to isotope $\g_i$ to $\cQ$, a contradiction to
the fact that $C$ links $\g_i$ in $\cB$. So $E\cap A'$ can only contain
simple closed curves.

If $\phi \subset E\cap A'$ is a simple closed curve, then
there is a disk $E'\subset E$ with $\partial E' =\phi$.
In particular, $\phi$ is nontrivial in $A'$, since otherwise we could
use a disk replacement argument to isotope $E$ removing $\phi$ from
$E\cap A'$. Since $\phi$ is isotopic to $C$ through $A'$, we could
obtain a disk in $Y$ with nontrivial boundary in $\partial N(C)$, a
contradiction. Thus, $A' \cap E = \es$.

If $\mu_1$ and $\mu_2$ lie on $\partial N(\g_1)$ and $\partial N(\g_2)$
respectively, then by an isotopy on
$\cQ \cup \partial N(\g_1 \cup \g_2)$, we can assume that $\mu_1$ and
$\mu_2$ are vertical arcs that do not wind around $\partial N(\g_1)$ or
$\partial N(\g_2)$. Then because $\be_1$ and $\be_2$ cannot cross the
equator $\partial A'\cap \cQ$, after possibly reindexing, $\be_1$
connects the top two punctures of $\cQ$ and $\be_2$ connects the bottom
two punctures. Since $\partial E$ must be trivial as an element of the
fundamental group of the handlebody $\cB \setminus N(\g_1 \cup \g_2)$,
there can be no twisting around the punctures, and $E$ must appear as
in Figure~\ref{disksII}~(a).

If $\mu_1$ and $\mu_2$ both lie on $\partial N(\g_1)$, then $\be_1$ and
$\be_2$ are loops on $\cQ$ based at a puncture and restricted to the
upper and lower hemisphere. Since no arcs in $A\cap\cQ$ can cut disks
off $\cQ$, then each $\be_1$ and $\be_2$ circle a puncture in $\cQ$.
Hence, $E$ must appear as in Figure~\ref{disksII}~(b).
A similar case occurs when $\mu_1$ and $\mu_2$ both lie on $\partial N(\g_2)$.

This argument allows us to introduce the following language.
If $\be \subset A\cap \cQ$ is any arc, then there is a unique
disk $E\subset A\cap \cB$ with $\be\subset \partial E$. If
$E$ is a Type I disk, we shall say that $\be$ is a Type I
arc. Otherwise, we will say that $\be$ is a Type II arc.

Next, we show that all intersection arcs in $A\cap \cQ$ are of the same type.
If $\G_1 \neq \G_2$, then, if there
exists a Type I disk, the two boundaries of $A$ are on different
components and only Type I disks can occur. 
If there is not a Type I disk, then all disks are Type II.
On the other hand,
if $\G_1 =  \G_2$, both components of $\partial A$
intersect $\partial N(\g_1)$ and $\partial N(\g_2)$ the same equal number of
times. In this case, the existence of a Type II disk $E_1$
with boundary intersecting $\partial N(\g_1)$ in two arcs,
implies that there exists a Type II disk $E_2$
intersecting $\partial N(\g_2)$. But $E_1$ and $E_2$
would then intersect, a contradiction.

%

Assume that all arcs in $A\cap \cQ$ are of Type I
and let $E$ be a connected component of $A\setminus \cB$. Then,
when we switch from $L'$ to $L$, $E$ can be extended to a 
disk properly embedded
in $M \setminus N(L)$. Thus, there is a trivial component in $L$, a
contradiction to its hyperbolicity.

Our next argument eliminates the last case when all intersections of
$A \cap \cQ$ are of Type II, and $\G_1 \neq \G_2$,
since we cannot mix the two types of Type II
intersections. Until the end of the proof we will assume that
$\partial A \subset \partial N(\G_1)$. Let $E\subset A$ be a
connected component of $A\setminus \cB$, and we label
$\partial E = \be_1\cup\mu_1\cup\be_2\cup\mu_2$, where $\be_1$ and
$\be_2$ lie in $\cQ$ and $\mu_1$ and $\mu_2$ lie in $\partial N(\G_1)$.
Then, $\mu_1$ and $\mu_2$ define two disks $\Delta,\wt{\Delta}$ in the
annulus $\partial N(\G_1)\setminus\cB$. We assume that the disk
$\Delta$ is the one that makes $\wt{A} = E\cup \Delta$ an annulus in
$Y\setminus \cB$ with both boundary components in $\cQ$ parallel to the
punctures that come from $\G_2$.

After capping $\wt{A}$ with the two once-punctured disks bounded by
$\partial \wt{A}$ in $\cQ$, we create an incompressible annulus
$\wh{A}$ in $Y\setminus \cB$ which also lives and is incompressible in
$X\setminus \cB$. Since $X$ is hyperbolic, it follows that $\wh{A}$
must be boundary-parallel to $\G_2$. But this implies that $\mu_1$ is
parallel in $X\setminus \cB$ to the arc $j_2 = \G_2\setminus \cB$,
and there exists a disk $E'\subset X\setminus \cB$ with
$\partial E' = \mu_1\cup \nu_1\cup j_2\cup \nu_2$, where
$\nu_1\cup \nu_2 = E'\cap \cQ$ are two arcs joining the respective
two upper punctures and the two lower punctures of $\cQ$ which avoid
the equator of $\cQ$. It then follows that $\nu_1\cup g$ and
$\nu_2\cup g'$ bound two respective disks in $\cB\setminus L$,
and the union of those disks with $E'$ is an essential disk
in $M\setminus L$, contradicting hyperbolicity of $X$ and
finishing the proof of Claim~\ref{cl:annuli}.
\end{proof}

\begin{claim}\label{clnotori2}
$Y$ does not admit essential tori.
\end{claim}
\begin{proof}
We argue by contradiction and
suppose that $T$ is a torus which is
incompressible and not boundary-parallel in $Y$.
First, suppose that $T \cap D =\es$. Then, after an
isotopy in $Y$,
we may assume that $T\cap \cB = \es$. Hence,
$T\subset X$ and, since $X$ is hyperbolic,
$T$ is either compressible or
boundary parallel in $X$.
If $T$ is boundary parallel, since both $G$ and $G'$ intersect
$\cB$, $T$ must be parallel to a component $J$ of $L$ which lives
in $L'$, contradicting that $T$ is essential in $Y$.

Next, we treat the case where $T$ is compressible in $X$;
let $E\subset X$ be a compression disk for $T$
and assume that $E$ has the least number of intersection curves with
$\cQ$ among compression disks for $T$.
Since $T$ is incompressible in $Y$,
$E$ intersects $\cB\cap L'$ and
the arc $\a$, which is a complete geodesic in the hyperbolic metric of $X$.
Let $\ol{N}(E)\subset X$
be a closed neighborhood of $E$ with
coordinates $E\times [0,1]$ and such that
$(\partial E \times [0,1])\subset T$.
Since $\a$ is transverse to $E$,
we may choose such a coordinate system on $\ol{N}(E)$
in such a way that,
for each $t\in[0,1]$, each component of $\a\cap \ol{N}(E)$
intersects $E\times \{t\}$ transversely
in a single point.

Let $S= (T\setminus (\partial E\times [0,1]))\cup (E\times\{0\})
\cup (E\times\{1\})$. Then, $S$ is a sphere in $X$
and $T\setminus S = \partial E\times[0,1]$.
Since $X$ is hyperbolic, $S$ separates
and must bound a closed ball
$B \subset X$.
Let $\a_1,\,\a_2,\,\ldots,\,\a_n$ be the arcs
in $\a\cap \ol{N}(E)$. We claim that each $\a_i$ is contained in $B$.
This follows because
the endpoints of $\a$ are in $L$, $L\cap B = \es$ and
$T\cap \cB = \es$. In particular, $\ol{N}(E) \subset B$.

\begin{figure}[ht]
\centering
\includegraphics[width=0.8\textwidth]{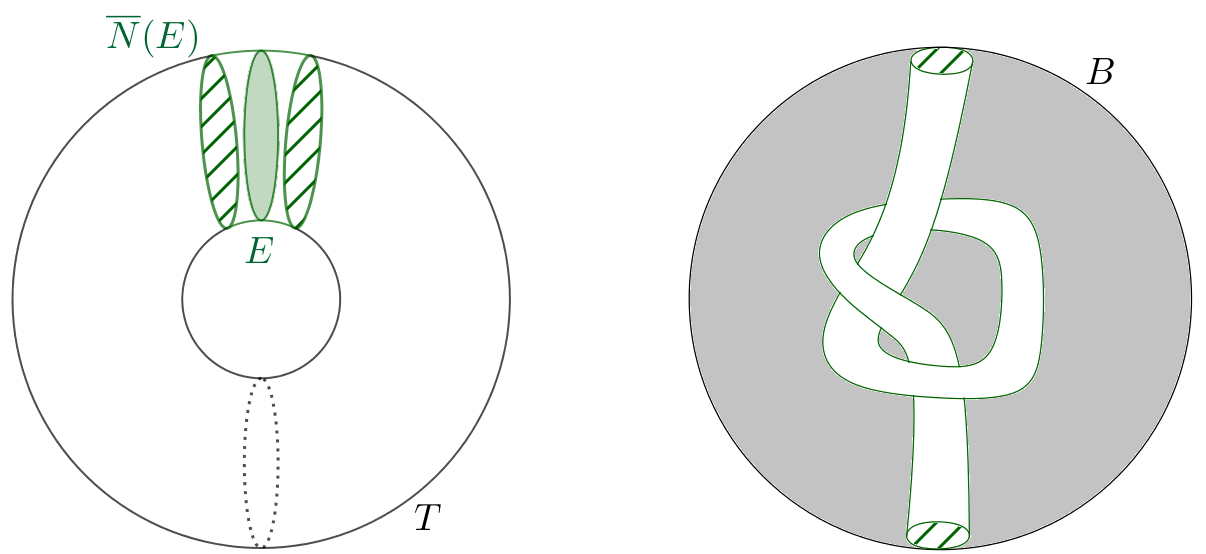}
\caption{
The figure on the left depicts the compressing disk
$E$ for the torus $T$ and the neighborhood $\ol{N}(E)$.
The figure on the
right shows the sphere $S$ bounding the ball $B$
and the (possibly knotted)
region $\ol{N}(E)\subset B$,
which defines the highlighted knot exterior
$W = B\setminus \ol{N}(E)$.}
\label{knotext}
\end{figure}

Let $W= B\setminus\ol{N}(E)$. Then
$\partial W = T$ and $W$ is a knot exterior in $B$
bounded by $T$ (in fact, we think of $W$ as obtained from $B$ by
removing a potentially knotted hole, see Figure~\ref{knotext}).
Since $T\cap L'= \es$ and $L'$ intersects $\ol{N}(E)$, then
$W \subset Y$.
Our next argument is to show that $\ol{N}(E)$ is unknotted in $B$;
thus $W$ is a solid torus bounded by $T$, which
contradicts the essentiality of $T$ in $Y$.

Let $\Pi\colon \hn3\to X$ be the Riemannian universal
covering map of $X$.
By appropriately choosing a neighborhood $N(L)$,
it follows that
$\Pi^{-1}(\partial N(L))$ is a
collection of horospheres in $\hn3$. Moreover,
$\Pi^{-1}(\a)$ is a collection of geodesics connecting
these horospheres. On the other hand,
$B$ lifts to a collection of balls, one of which is a ball
$\widetilde{B}$, containing a lift of $W$, denoted $\widetilde{W}$.

In order for $T$ to be incompressible in
$Y$, a lift of $\alpha$, which we denote by
$\widetilde{\alpha}$, must
pass through the hole $\wt{B}\setminus \wt{W}$ in $\wt{B}$.
Since $\widetilde{\alpha}$ is a
geodesic in $\mathbb{H}^3$, it follows that $\wt{\a}$ is unknotted,
which implies that $\wt{W}$ is a solid torus. Since $W$ is homeomorphic
to $\wt{W}$, this gives a contradiction, as previously explained.

It remains to prove that there are no essential tori in
$Y$ which intersect $D$.
Arguing by contradiction, assume that
$T$ is such a torus.
Since $D$ is incompressible in $Y$,
a disk replacement argument allows us to further
assume that any curve in $T\cap D$
is nontrivial both in $T$ and in $D$.
Let $\be$ be a curve in $T\cap D$.
If $\be$ encircles one puncture of $D$, take an innermost
curve in such intersection and use the one-punctured disk
it bounds in $D$ to surger $T$ and obtain an essential annulus
in $Y$ with boundary in $\partial N(\G_i)$.
If $\be$ encircles both punctures of $D$,
take an outermost curve on $T\cap D$ and use the outer annulus
on $D$ to surger $T$ and obtain an essential annulus with boundary
in $\partial N(C)$. Since Claim~\ref{cl:annuli}
gives that $Y$ does not admit essential annuli, this proves
Claim~\ref{clnotori2}.
\end{proof}

Thus, having proved that $Y$ satisfies Thurston's hyperbolicity
conditions, Theorem~\ref{switch} follows when $M$ is orientable.

Next, we assume that $M$ is nonorientable and that $L,\,L',\,\a$ and
$\cB$ are as before. Let $\Pi\colon \wh{M}\to M$ be the two-sheeted
oriented covering map of $M$. Then,
$\wh{L} = \Pi^{-1}(L)$ is a hyperbolic link in $\wh{M}$. We also
let $\wh{L}''=\Pi^{-1}(L')$, $\cB_1,\,\cB_2$ be the connected
components of $\Pi^{-1}(\cB)$ and $\a_1,\,\a_2$ be the connected
components of $\Pi^{-1}(\a)$. We claim that
$\wh{Y} = \wh{M}\setminus \wh{L}''$ is also hyperbolic.
Note that, as explained
in the proof of the nonorientable setting for the chain move,
the fact that $\wh{Y}$ is hyperbolic implies that
$Y = M\setminus L'$ is hyperbolic.

Since $\a$ is a complete geodesic in the hyperbolic metric of
$M\setminus L$, then both $\a_1$ and $\a_2$ are complete geodesics in
$\wh{M}\setminus \wh{L}$. In particular, since $\wh{M}$ is
orientable, the switch move allows us to
replace $\wh{L}\cap \cB_1$ by a tangle diffeomorphic
to $L'\cap \cB$ to create a new hyperbolic link $\wh{L}'$ in $\wh{M}$.
Note that $\wh{L}''$ may be obtained from $\wh{L}'$ by replacing
the tangle $\wh{L}'\cap \cB_2 = \wh{L}\cap \cB_2$ by a
tangle diffeomorphic to $L\cap \cB$.

Since it might be the case
that $\a_2$ is not isotopic to a geodesic in the hyperbolic metric
of $\wh{X} = \wh{M}\setminus \wh{L}'$, one cannot directly
apply the switch move
a second time. However, most of the arguments in its proof can
be repeated without change for this setting. We next
guide the reader over the steps in the proof that need some
adaptation.

First, the arguments
in the proof of the orientable case for the switch move
can be used to prove that the four-punctured sphere
$\cQ_1 = \partial \cB_1\setminus \wh{L}''$ is incompressible and
boundary-incompressible in $\wh{X}$ and in
$\wh{Y} = \wh{M}\setminus \wh{L}''$, that
$\cQ_2 = \partial \cB_2\setminus \wh{L}''$ is
incompressible and boundary-incompressible in
$\wh{Y}$ and that $\wh{Y}$
does not admit any essential disks and essential spheres.

To prove that $\wh{Y}$ does not admit any essential annuli,
the arguments in Claim~\ref{cl:annuli} apply to
show that if $A$ is an essential annulus in $\wh{Y}$, then
both boundary components of $A$ are meridians in
a component $\wh{G}'$ of $\wh{L}'$ that
intersects $\cB_2$ and that
we may isotope $A$ in $\wh{Y}$ to assume that
$A\cap \cB_2 = \es$.
Since $\wh{X}\setminus \cB_2 = \wh{Y}\setminus \cB_2$,
then $A$ is an incompressible annulus in $\wh{X}$,
and $A$ must
be boundary-parallel in $ \wh{M}\setminus N(\wh{L}')$.
In particular, after an isotopy in $\wh{Y}$ that does
not change the property $A\cap \cB_2 = \es$,
we may assume that $\partial A \cap \cB_1 = \es$ and
that if $A$ intersects $\cB_1$, then each connected component
of $A\cap \cB_1$ is an annulus parallel to one of the
two arcs in the tangle $\wh{L}'\cap \cB_1$.

If $A\cap \cB_1 = \es$, $A$ is an incompressible annulus
in the hyperbolic manifold $\wh{M}\setminus \wh{L}$,
and the same arguments in the proof of Claim~\ref{cl:annuli}
apply to show that the neighborhood
through which $A$ is parallel to an annulus $A'$
in $\partial N(\wh{L})$ can be capped off by
meridianal disks to define a ball $W'$ in $\wh{M}$
that contains both $\cB_1$ and $\cB_2$ and may be used to homotope
the arcs $\a_1$ and $\a_2$ to $\partial N(\wh{L})$, a
contradiction with the fact that both $\a_1$ and $\a_2$
are geodesics in the hyperbolic
metric of $\wh{M}\setminus \wh{L}$.

Hence, there must exist 
$A_0$ a connected component of $A\cap \cB_1$. We
assume that $A_0$ is innermost in the sense that
no other component of $A\cap \cB_1$ lies in the ball region
defined by $A_0$ in $\cB_1$.
Then, each boundary component of $A_0$ is a curve in $\cQ_1$
that encircles one puncture, defining a
once-punctured disk in $\cQ_1$. Using these
two once-punctured disks to surger $A$
gives two incompressible annuli in $\wh{M}\setminus N(\wh{L}')$,
both disjoint from $\cB_2$ and at least one of them must be
essential in $\wh{Y}$. By induction on the number
of components in $A\cap \cB_1$, this argument
yields an essential annulus $\wh{A}$ in $\wh{Y}$, with both boundary
components being meridians, and that is disjoint both from
$\cB_1$ and from $\cB_2$. As already shown, this is a contradiction
that proves that $\wh{Y}$ does not admit any essential annuli.

The proof that $\wh{Y}$ does not admit any essential tori uses the
arguments in Claim~\ref{clnotori2}. Among all possible
essential tori, the only case
that still needs an adaptation is when $V$ is
an essential torus in $\wh{Y}$ that can
be isotoped to be disjoint from $\cB_2$.
Let $D_1$ be a twice punctured disk
in $\cB_1\setminus \wh{L}'$ bounded by the trivial
component $\wh{C}_1$ of $\wh{L}'\cap \cB_1$. Then,
$D_1$ is incompressible and we may isotope $V$ in $\wh{Y}$
to assume that there are no trivial curves in $V\cap D_1$.
Hence, $V\cap D_1 = \es$, since the existence of a nontrivial
curve in $V\cap D_1$
allows us to surger $V$ to produce an essential annulus in
$\wh{Y}$, which we already proved that cannot exist.
In particular, $V$ can also be isotoped in $\wh{Y}$ to be disjoint
from $\cB_1$, and then $V$ is a torus in the
hyperbolic manifold $\wh{M}\setminus \wh{L}$.
Since $V$ cannot be boundary parallel in $\wh{M}\setminus \wh{L}$,
there exists a compressing disk $E$ for $V$
in $\wh{M}\setminus \wh{L}$, and the fact that $V$ is incompressible
in $\wh{Y}$ implies that $E$ must necessarily intersect the arcs
$\a_1$ and $\a_2$, which are geodesics in the hyperbolic metric of
$\wh{M}\setminus \wh{L}$. Now, the same arguments in the proof
of Claim~\ref{clnotori2} apply to show that $V$ bounds a unknotted
solid region $W$ in $\wh{Y}$, contradicting the fact that $V$
is essential in $\wh{Y}$.
This argument finishes the proof that $\wh{Y}$ satisfies
Thurston's hyperbolicity conditions and,
as already explained,
proves the Switch Move Theorem for the nonorientable case.
\end{proof}

\section{The Switch Move Gluing Operation.}
\label{secGlue}

We describe in Theorem~\ref{thmoperation} below
a method to obtain new hyperbolic 3-manifolds
of finite volume from two previously given ones;
this method uses a variant of the switch move (Theorem~\ref{switch}).
Before stating this result, we set the notation.
Let $M_1,\,M_2$ be compact 3-manifolds with nonempty boundary with
zero Euler characteristic.
Let $L_1\subset M_1$ and $L_2\subset M_2$ be links
such that int$(M_1\setminus L_1)$ and
int$(M_2\setminus L_2)$ admit hyperbolic metrics
of finite volume.
Let $T_1\subset \partial M_1$ and $T_2\subset \partial M_2$
be boundary components that are either both tori or both Klein
bottles. For each $i\in \{1,2\}$, let $\a_i$
be a complete geodesic in the hyperbolic metric of
int$(M_i\setminus L_i)$ with one endpoint in $T_i$
and another endpoint in a component $J_i$ of $L_i$.

Choose a gluing diffeomorphism $\phi\colon T_1\to T_2$ which
maps the endpoint of $\a_1$ in $T_1$ to the endpoint of $\a_2$ in $T_2$ and
let $M$ be the quotient manifold obtained by gluing $M_1$ and $M_2$
along $\phi$. Then, we may consider
$M_1$ and $M_2$ as subsets of $M$, which are separated by
the compact surface $T$ that comes from the identification
of $T_1$ and $T_2$.

\begin{theorem}[Switch Move Gluing Operation]\label{thmoperation}
With the above notation, let $\a$ be the concatenation of
$\a_1$ and $\a_2^{-1}$ in $M$.
Let $\cB$ be a ball neighborhood of $\a$ in $M$ that intersects
$L = L_1\cup L_2$ in two arcs $g,g'$
as in Figure~\ref{switchbefore} and intersects
$T$ in a disk $\Delta$. Let $L'$ be the resulting link
in $M$ obtained by replacing $g\cup g'$ by the tangle
$\g_1\cup\g_2\cup C$ as in Figure~\ref{figaugmented}~(b), where
$C\subset \Delta$. Then, the manifold $Y = M\setminus L'$ is hyperbolic.
\end{theorem}

After choosing $\phi$ as above,
like in the case of the switch move, the operation
described above may yield two distinct
hyperbolic 3-manifolds depending on the projection of the
strands $g,g'$, see Remark~\ref{notwelldefined}.

\begin{proof}[Proof of Theorem~\ref{thmoperation}]
We first prove the theorem in the case when $M$ is
orientable.
Let $X = M\setminus L$. Then, the setting in
Theorem~\ref{thmoperation} is the same as in the
Switch Move Theorem, Theorem~\ref{switch}, with the exception
that $X$ is no longer hyperbolic. However, $X$ is close to being
hyperbolic in the following sense:

\begin{claim}\label{clalmosthyp}
$X$ does not admit any essential spheres, essential disks and
essential annuli. Moreover, any essential torus in $X$ is isotopic
to $T$.
\end{claim}
\begin{proof}[Proof of Claim~\ref{clalmosthyp}]
Suppose there were an essential disk $E$ in $M \setminus L$.
Since both int$(M_1 \setminus L_1)$ and int$(M_2 \setminus L_2)$
are hyperbolic,
it follows that $E$ must intersect $T$.
But because $T$ is incompressible and $\partial E$ is disjoint from $T$,
we may replace subdisks in $E$ by disks in $T$
to obtain an essential disk in either
int$(M_1 \setminus L_1)$ or int$(M_2 \setminus L_2)$, a contradiction.
The same argument shows that an essential sphere in $M \setminus L$
would generate an essential sphere in either
int$(M_1 \setminus L_1)$ or
int$(M_2 \setminus L_2)$, also a contradiction.

Since $T$ separates,
an essential torus in $M \setminus L$ that does not intersect
$T$ must be parallel to $T$, and hence isotopic to $T$.
To prove that $T$ is the only possible essential torus up
to isotopy,
we argue by contradiction and assume that
$\wh{T}$ is an essential torus in $M\setminus L$
that is not isotopic to $T$ and has the fewest number of intersection
components with $T$. Then,
any curve in $\wh{T}\cap T$ is nontrivial both
in $T$ and in $\wh{T}$. Then, it follows that
there is a component of $\wh{T}\setminus T$
that is an essential annulus either in
int$(M_1 \setminus L_1)$ or in int$(M_2 \setminus L_2)$, a contradiction.
Analogously, we may show that $M\setminus L$ does not admit any
essential annuli, and this proves Claim~\ref{clalmosthyp}.
\end{proof}

Having proved Claim~\ref{clalmosthyp},
observing that $L\setminus \cB$ is not a rational
tangle in a 3-ball, we notice that
the arguments in the proof of the Switch Move Theorem
apply directly to show that $Y$ does not admit any essential spheres
and any essential disks and that
the four-punctured sphere
$\cQ = \partial \cB\setminus L = \partial \cB\setminus L'$
and the twice-punctured
disk $D$ bounded by $C$ on $T$ are incompressible and
boundary-incompressible in $Y$.

To prove that $Y$ does not admit any essential annuli,
the proof of Claim~\ref{cl:annuli} applies directly, since
the arcs $g$ and $g'$ in $L$ that intersect $\cB$
are on distinct components of $L$.
Hence, to prove Theorem~\ref{thmoperation}
when $M$ is orientable, it remains to show that $Y$
does not admit any essential tori.

We argue by contradiction and assume that $V$ is an essential torus
in $M\setminus L'$ that has the least number of intersection components
with $T$ among all essential tori in $Y$. Then, after assuming
general position, $T\cap V$ is
a finite collection of pairwise disjoint simple closed curves.
Let $\g$ be one of such intersection components. If $\g\subset D$,
then it does not bound a disk in $T$ and either encircles
one or two of the punctures of $D$. Then, we can choose
a component $\g'$ in $V\cap D$
(if $\g$ encircles one puncture, we choose $\g'$ as
an innermost curve, otherwise we choose $\g'$ as an
outermost curve) and surger $T$ to obtain an essential annulus in $Y$,
a contradiction. It then follows that $V\cap D = \es$,
and then $V$ can be isotoped through $Y$
to be disjoint from $\cB$, without increasing the number
of intersection components in $V\cap T$. Hence,
$V$ is a torus that is contained in $X$.

In $X$, $V$ is not isotopic to $T$,
since $V\subset M\setminus L'$ and any torus isotopic to
$T$ is punctured by $L'$. We claim that $V\cap T = \es$. Argue
by contradiction and assume that there exists a
curve $\g$ in $V\cap T$. Then, $\g$ does not intersect $D$
and there are two possibilities: either $\g$ is a nontrivial curve in
$T$ or $\g$, together with $C$, bounds an annulus in $T\setminus D$.
In the latter case, we may use this annulus to surger $V$ and obtain
an essential annulus in $Y$. Since $Y$ does not admit essential
annuli, $\g$ is nontrivial
in $T$. Then, there is a component of $V\setminus T$ that is an
essential annulus in either
$M_1\setminus L_1$ or in $M_2\setminus L_2$, which
proves that $V\cap T = \es$.

Since $T$ separates $M$,
there exists $i\in\{1,2\}$ such that $V$ is a torus in
the hyperbolic manifold $M_i\setminus L_i$.
For convenience, we will assume that $i=1$.
Note that $V$ cannot be boundary parallel in $M_1\setminus L_1$,
since this either contradicts its essentiality in
$Y$ or the fact that it is not isotopic to $T$. Hence, it must
be the case that $V$ is compressible in $M_1\setminus L_1$.
Let $E\subset M_1\setminus L_1$ be a compression disk for $V$. Then,
since $V$ is incompressible in $Y$, the geodesic $\a_1$ must intersect
$E$. Now, the same arguments used in the proof of Claim~\ref{clnotori2}
for the case when $T$ was an essential torus in
$Y$, disjoint from $\cB$ and compressible in $X$
apply to obtain that $V$ is compressible in $Y$, a contradiction
that proves Theorem~\ref{thmoperation} when $M$ is orientable.

Next, we sketch the arguments that prove
Theorem~\ref{thmoperation} when $M$ is nonorientable,
using the notation already introduced.
There are two cases to consider.
If $T$ is a torus in $M$
and $\Pi\colon \wh{M}\to M$ is the oriented two-sheeted
covering map, then $T$ lifts to a separating
torus $\wh{T}$ in $\wh{M}$ and the orientable case
can be applied directly. On the other hand, if $T$ is a Klein
bottle in $M$, then $T$ does not lift, but
$\wh{T} = \Pi^{-1}(T)$ is again a separating torus in $\wh{M}$.
Now, there are two balls $\cB_1,\,\cB_2$ which are
the connected components of $\Pi^{-1}(\cB)$ and the proof of
the Switch Move Theorem when $M$ is nonorientable can
be repeated
to prove Theorem~\ref{thmoperation}
when $M$ is nonorientable.
\end{proof}

\appendix
\section*{Appendix: Rational tangle case for the chain move}\label{rationaltangles}

In this section we consider the special case relevant to the chain move
when the underlying manifold $M$ is the 3-sphere $S^3$ and the tangle
outside of the ball $\cB$ is a rational tangle. We represent rational
tangles by a sequence of integers, all nonpositive or all nonnegative,
as in Chapter 2 of~\cite{Adbook}.

\begin{proof}[Proof of Lemma~\ref{rational}]
We first untwist the $k$ half-twists of $C_2$ in $\cB$ and thereby
twist $-k$ half-twists in $R$ to move this integer tangle into $R$,
which will remain a rational tangle. Let $\cB' = S^3\setminus \cB$ be
the ball containing the rational tangle $R$ to the outside of $\cB$.
We let $C_3$ be the component of $L$ containing $\gamma_1$ and $C_4$ be
the component containing $\gamma_2$, keeping in mind that it may be the
case that $C_3 = C_4$. 
Since $R$ is a rational tangle, there are no other components in
$L$. Our goal will be to show that, after this twisting,
$S^3\setminus L$ is hyperbolic unless $R$ is one of the four tangles
depicted in Figure~\ref{rationalcounterexamples}.

\begin{figure}[htpb]
\begin{center}
\includegraphics[width=1.0\textwidth]{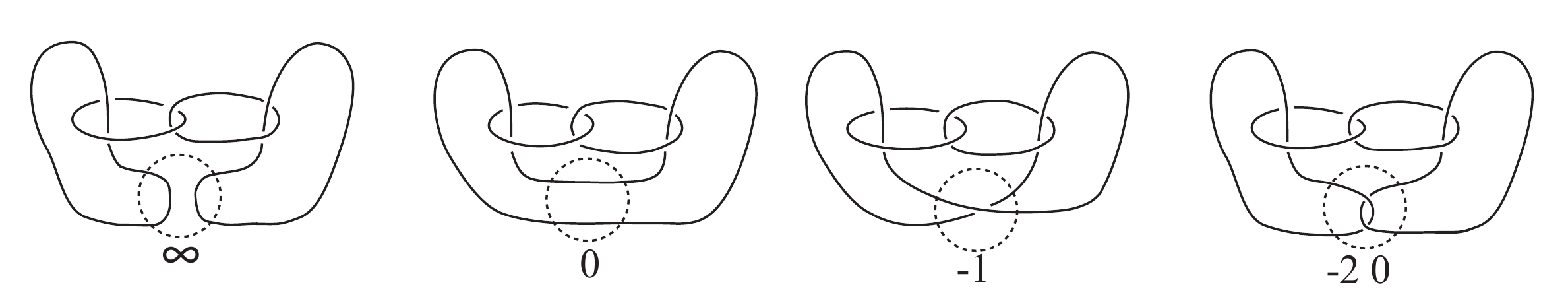}
\caption{Non-hyperbolic link complements.}
\label{rationalcounterexamples}
\end{center}
\end{figure}

We first show that there are no essential spheres in $S^3\setminus L$.
Let us assume that $E$ is an essential sphere in $S^3\setminus L$. Then,
$E$ separates $S^3$ into two balls, each of which contains at least one
component of $L$. Since $C_1$ is linked with $C_2$ and also with $C_3$,
these three components must be on the same side $V$ of $E$. But $C_2$ is
also linked with $C_4$ so $C_4\subset V$, contradicting the need for components
to lie on both sides of $E$.

Next, we prove that there are no essential disks. If there were such a
disk $D$ with boundary in a component $C$ of $L$, then the boundary
of $N(D \cap C)$ is a sphere that separates $C$ from the other components
of $L$, contradicting that there are no essential spheres.

We now show that the existence of an essential annulus in $S^3\setminus L$
gives that $L$ is one of the four exceptions in Figure~\ref{rationalcounterexamples}.
Suppose such an annulus $A$ exists. If the boundary curves of $A$ lie on distinct
components $C_i,\,C_j$ of $L$, then we may take $A' = \partial N(A \cup \partial N(C_j))$,
which is an essential annulus in $S^3\setminus L$ with
boundary curves in $\partial N (C_i)$. Hence, it suffices to show that there
are no essential annuli that have both their boundary components on the same $\partial N(C_i)$.

Let $D_1$ and $D_2$ be the twice-punctured disks bounded by $C_1$ and $C_2$
respectively and let $\cQ$ be the four-punctured sphere 
$\partial \cB\setminus L$.
We assume that $A$ is an essential annulus with both boundary curves on the
same component $\partial N(L)$ with the minimal number $n\geq0$ of
intersections with $D_1$, $D_2$ and $\cQ$ among all such annuli.

We first prove that $A$ can be chosen so that $A\cap (D_1\cup D_2)$ does not
contain any simple closed curve. Let us assume that $\a$ is a simple closed
curve in $A\cap D_i$ and that $\a$ encloses only one puncture of $D_i$,
which comes from a component $C_j$ of $L$. In this case, after possibly
replacing $\a$ by an innermost curve, we could surger $A$ to obtain two
annuli $A_1,\,A_2$, at least one of which, is essential and with
fewer intersection components with $D_1\cup D_2\cup \cQ$. Hence, the minimality
property defining $A$ implies that the boundary curves of $A_1$ lie on distinct components
of $\partial N(L)$. In particular, $A_2$ cannot be boundary parallel,
from where it follows that $A_2$ is also essential. 
Hence, the roles of $A_1$ and $A_2$ are identical and at least
one of them, say $A_1$, satisfies that the number
of intersection curves with $D_1\cup D_2\cup \cQ$ is less than $n/2$.
Let $A' = \partial N(A_1\cup \partial N(C_j))$. Then, $A'$ is an
essential annulus in $S^3\setminus L$ with both boundary curves in the same
component of $\partial N(L)$. However, the number of intersection curves
in $A'\cap (D_1\cup D_2\cup \cQ)$ is less than $n$, a contradiction.
On the other hand, if $\a$ encircles both punctures of $D_i$ and it is
outermost, we could again surger $A$ to obtain two annuli with longitudinal
boundaries, at least one of which would be essential, and the same argument
as before applies to give a contradiction.

We next prove that any intersection arc $\alpha$ in $A \cap D_i$ is not
boundary-parallel on either $A$ or $D_i$. If it were boundary-parallel
on both, we could form a disk $D$ from the disk $\a$ cuts from $A$ and
the disk $\a$ cuts from $D_i$. Then the boundary of $D$ must be a trivial
curve in $\partial N(C_i)$, so we can add to it a disk from $\partial N(C_i)$
to obtain a sphere that must bound a ball. Hence we can isotope $A$ through
this ball to remove the intersection arc $\a$, a contradiction to the
minimality of intersection arcs.

If $\alpha$ is boundary-parallel on $D_i$ but not on $A$, then $A$ is
boundary-compressible, a contradiction. If $\alpha$ is boundary-parallel
on $A$ but not on $D_i$, then $\alpha$ splits the punctures on $D_i$,
since both boundary components of $A$ lie on the same component of $\partial N(L)$.
Let $D$ be the disk $\alpha$ cuts from $A$. Then, $\partial D$ has nonzero
linking number with at least one of the components that puncture $D_i$,
a contradiction that proves that $\a$ is not boundary-parallel on either $A$ or $D_i$.

We next treat the case where the boundary components of $A$ are $(p,q)$-curves
on $\partial N(C_1)$. Note that by symmetry, the roles of $C_1$ and $C_2$
can be reversed. If $\abs{p}\geq1$, there are $\abs{p}$ arcs in $A\cap D_1$,
each separating the two punctures of $D_1$ and joining one boundary component
of $A$ to the other. In particular, all such arcs are parallel on $D_1$ and also on $A$.

Next, consider $A$ in $V = S^3 \setminus N(C_1)$, which is a solid torus.
Since $\partial A$ is nontrivial on $\partial V$ and $A$ is properly embedded
on $V$, then $A$ must be boundary-parallel on $V$ and define a solid torus
region $V'\subset V$. In particular, since we have minimized the intersections
of $A$ with $D_1$, which is a meridianal disk for $V$, then $V'\cap D_1$
has $\abs{p}$ components, each of which is a disk with boundary given by
two arcs, one in $\partial D_1$ and the other on $A$. In particular, the fact
that all such arcs are parallel on $D_1$ implies that $\abs{p}\leq 2$.

We now treat the case when $\abs{p} = 2$. In this case, any closed curve
in $V'$ has even linking number with $C_1$, which implies that both $C_2$
and $C_3$ lie outside of $V'$. But, since $A$ is essential in $S^3\setminus L$,
$V'$ must contain at least one of the components of $L\setminus C_1$,
hence $C_4$ is distinct from $C_3$ and $C_4\subset V'$. But then, the fact
that $C_2$ links $C_4$ and $C_4$ is separable from $C_1$ implies
that $C_2 \subset V'$, a contradiction.

Next assume that $|p| = 1$, and we will show that $q = 0$ or $\abs{q} = 1$.
Arguing by contradiction, suppose $\abs{q} > 1$. As before, at least one
component of $L\setminus C_1$ must be inside $V'$ to prevent $A$ from being
boundary-parallel in $S^3\setminus L$.
First assume that $C_2$ is in $V'$ and we will show that this implies that
$L\setminus C_1\subset V'$. Since the linking number of $C_1$ with $C_2$ is one,
$C_2$ must intersect every meridianal disk of $V'$. But then, again by linking
number, the winding number of $C_2$ in $V'$ must be 1. Suppose $C_3 \not\subset V'$.
Since $C_3$ links $C_1$ once and $V'$ is a solid torus glued to $N(C_1)$
along a $(1,q)$-curve, $C_3$ links $C_2$ $q$ times, a contradiction that
shows that $C_3$ is also in $V'$. Suppose $C_4\neq C_3$ and $C_4 \not\subset V'$.
Then, since $C_4$ does not link $C_1$, once again one can prove that it cannot
link $C_2$, a contradiction. Hence, $L\setminus C_1 \subset V'$. An analogous
argument also shows that starting from any component of $L\setminus C_1$ inside $V'$,
all of them must be inside $V'$. But then the single arc of intersection
of $A$ in $D_1$ cannot separate the punctures, as both components
corresponding to the punctures are not separated by $A$, a contradiction
that shows that $\abs{q}\leq1$.

If $q = 0$, both boundary components of $A$ are meridians on $\partial N(C_1)$
and $A\cap D_1$ is a single arc that must split the two punctures.
We may consider $A$ as a sphere punctured twice by $C_1$, with $C_3$ to one side
and $C_2$ to the other, and since $C_4$ is linked with $C_2$, and therefore must
be to that same side, it must be the case that $C_3$ and $C_4$ are distinct and
separable. Hence the rational tangle must be the tangle denoted by $\infty$
appearing on the left in Figure~\ref{rationalcounterexamples}.

Now, we consider the case for $\partial A \subset \partial N(C_1)$ when
$|p|= 1$ and $|q| = 1$. Let $V''$ be the closure of the complement
of $V'$ in $V$ and note that, in $V$,
$A$ is parallel to both $V'\cap \partial N(C_1)$ and to $V''\cap \partial N(C_1)$.
Without loss of generality, we will assume that $C_2\subset V'$.
Then,
since $C_1 \cup C_2$ is a Hopf link, it must be that $C_2$ is isotopic to the
core curve of $V'$. In particular, $C_2$ is isotopic, in $V'$, to either component
of $\partial A$ on $\partial N(C_1)$, and $C_2$ must hit every meridianal disk
of $V'$ and have winding number 1. Note that $V''' =V' \cup N(C_1)$ is a solid
torus such that each of  $C_1$ or $C_2$ can be taken to be a core curve.

We next treat the case where $C_3 \neq C_4$. Since $C_4$ links
$C_2$ but not $C_1$, it cannot be in $V''$.
Since $C_3$ links $C_1$ but not $C_2$, it also cannot be in $V''$. 
But then $A$
is isotopic to $\partial N(C_1)$ through $V''$, a contradiction.

Thus, it must be the case that $C_3 = C_4$. Since $C_2\subset V'$, then $C_3\subset V''$,
since otherwise $A$ would be boundary-compressible. Moreover, $C_3$ also must hit
every meridianal disk of $V''$ and have winding number 1 so it links with $C_1$ once
and with $C_2$ once. In particular, because $C_3$ is prime (every rational knot is
a 2-bridge knot, therefore prime), it must be that $C_3$ is isotopic to the core
curve of $V''$ and is in particular, trivial in $S^3$. This forces the rational
tangle $L\setminus \cB$ to be an integer tangle. But for every such tangle other
than 0 and -1 (cases excluded in Figure~\ref{rationalcounterexamples}), the
resultant 3-chain link is alternating and hence hyperbolic by~\cite{mena}.

In the final case, when $p = 0$ and $\abs{q} = 1$, both boundary components of $A$ are
longitudes on $\partial N(C_1)$. Neither $C_2$ nor $C_3$ can be in $V'$ since they
both link $C_1$ once. Hence, it must be the case that $C_3 \neq C_4$, and $C_4 \subset V'$.
Since $C_4$ must prevent $A$ from being boundary-parallel, it must hit every
meridianal disk in $V'$. But since $C_2$ links each of $C_1$ and $C_4$ once,
$C_4$ must have winding number 1 in $V'$. In order that $C_3 \cup C_4$ be a
rational and hence a 2-bridge link, it must be the case that $C_4$ intersects
a meridianal disk of $V'$ exactly once, implying that in fact $C_4$ is parallel
to the core curve of $V'$ and therefore parallel to $C_1$, which yields an
annulus $E$ with $\partial E = C_1\cup C_4$. Therefore, $C_3$ and $C_4$ are
also linked once, and in the same way that $C_3$ is linked with $C_1$. But
then, the fact that $C_3$ is not linked with $C_2$ implies that $C_3$ and $C_2$
are also parallel to each other, hence the rational tangle $L\setminus \cB$ is -2 0,
as in the fourth  case depicted in Figure~\ref{rationalcounterexamples}.
This completes the case when $\partial A \subset \partial N(C_1)$.

\medskip

Our next argument analyzes the situation when both boundary components of $A$
are $(p,q)$-curves on $\partial N(C_3)$. Then, there are two cases to treat
according to whether $C_3 \neq C_4$ or $C_3 = C_4$.

When $C_3 \neq C_4$, then both $C_3$ and $C_4$ are trivial as knots in $S^3$.
Once again, $A$ must be boundary-parallel in $V = S^3 \setminus N(C_3)$
and we let $V'$ denote the solid torus region that $A$ cuts from $V$.
Since $A$ is assumed to be essential, one of the other components $C_1,\,C_2,\,C_4$
of $L$ must be contained in $V'$. Next, we assume $|p| \geq 1$.

If $C_1 \subset V'$, then $C_1$ must hit every meridianal disk of $V'$
since it has linking number 1 with $C_3$. But since this linking number
with $C_3$ is $p k$, where $k$ is the winding number of $C_1$ in $V'$,
it follows that $|p| = 1$ and $k = 1$. Since $C_2$ links $C_1$, and $C_2$
is separable from $C_3$, it must be the case that $C_2 \subset V'$ also.
But $C_2$ does not hit every meridianal disk of $V'$ and hence $C_2$ is in
a ball contained in $V'$. But $C_4$ must link $C_2$, so $C_4 \subset V'$.
But then $L\setminus C_3\subset V'$ and $\abs{p} =1$, so the annulus $A$ is
boundary-compressible through the disk bounded by $C_3$, a contradiction.

The case when $C_2\subset V'$ implies that $C_4\subset V'$, as shown in the
previous paragraph. Hence, it suffices to get a contradiction assuming
$C_4 \subset V'$ and $C_1\not\subset V'$. Assuming we are not in the case of
the rational tangle being $\infty$, $C_4$ must intersect all meridianal disks
of $V'$. Since $C_2$ links $C_4$ but is separable from $C_3$, then $C_2 \subset V'$.
But $C_2$ cannot hit every meridianal disk of $V'$, and hence $C_2$ lives in a ball
in $V'$. Then, because $C_1$ links $C_2$, $C_1$ must also be in $V'$, a contradiction.

Thus it must be the case that $|p|=0$, and $\partial A$ is a pair of longitudes
on $\partial N(C_3)$. Then, because $C_1$ links $C_3$, it follows that $C_1$
cannot be in $V'$. If $C_2 \subset V'$, then, if $C_2$ misses a meridianal disk,
it is contained in a ball in $V'$, making it impossible for $C_1$ to link with it.
Hence, $C_2$ must hit every meridianal disk of $V'$. Since $C_1$ links each of $C_3$
and $C_2$ once, $C_2$ must have winding number 1 in $V'$. In fact,
because $C_1 \cup C_2$ is the Hopf link, $C_2$ must be  isotopic to the core
curve of $V'$, and hence isotopic to $C_3$. In particular, $C_4$ links $C_3$ the
same way it links $C_2$, and therefore the rational tangle is -2 0.

\medskip

Next, assume that $\partial A \subset \partial N(C_3)$ and $C_3 = C_4$.
Then, $C_3$ is a rational knot that is hyperbolic unless $C_3$ is trivial or
is  a 2-braid knot. This follows from the fact all rational knots are alternating
and the work in~\cite{mena}.

If $C_3$ is a nontrivial 2-braid knot, then, up to isotopy, there is a unique
annulus that is essential in $S^3 \setminus C_3$ corresponding to the complement
of $C_3$ in the torus upon which a 2-braid can be pictured to sit. But because both
$C_1$ and $C_2$ are linked once with  $C_3$, they both must puncture this annulus.
Thus any essential annulus in $S^3\setminus N(L)$ must be boundary-parallel in $S^3\setminus N(C_3)$.

If $C_3$ is nontrivial and is not a 2-braid knot, it has hyperbolic complement.
Therefore, in this situation also, since $\partial A$ is nontrivial on $\partial N(C_3)$,
$A$ must be boundary-parallel in $S^3\setminus N(C_3)$.

In either case, one of $C_1$ or $C_2$, say $C_1$, must be contained in the solid torus
$V'$ that $A$ cuts from $S^3 \setminus N(C_3)$.  Because $C_1$ links $C_3$ once, then $C_1$
must hit every meridianal disk of $V'$. But then, it must be the case that $|p| = 1$, and
the winding number of $C_1$ in $V'$ is 1. But if $|q| \geq 1$, then $C_1$ is a satellite
knot of $C_3$, which because $C_3$ is nontrivial, must be nontrivial as well, a contradiction.

So it must be the case that if $C_3$ is not trivial, both boundary components
of $A$ are meridians of $C_3$. Since it is assumed that there are no simple
closed curves of intersection of $A$ with $D_1$ or with $D_2$, then $A$ is disjoint from $D_1\cup D_2$.

But once $A$ does not intersect $D_1 \cup D_2$, we can push $A$ out of $\cB$ entirely.
But all twice-punctured spheres in a rational tangle are boundary-parallel. So $A$
is boundary-parallel in the complement of $\cB$, a contradiction.

Thus, $C_3$ is a trivial knot. Hence, the rational tangle must be represented
as an integer tangle. But this makes $L$ into a 3-chain and as in~\cite{NR}
(or as it also follows from~\cite{mena}), a 3-chain is non-hyperbolic if and
only if it is non-alternating, hence $L\setminus \cB$ is one of the tangles $0$
or $-1$ represented by the two middle figures in Figure~\ref{rationalcounterexamples}.
This finishes the proof that if $S^3\setminus L$ admits an essential annulus, then
$L$ is one of the links shown in Figure~\ref{rationalcounterexamples}.

\medskip

To finish the proof of Lemma~\ref{rational}, suppose $T$ is an essential torus
in $S^3\setminus L$. We may assume that $T$ has been isotoped to minimize the number
of intersection curves $\cQ$, and also
with the disks $D_1$ and $D_2$ bounded by $C_1$ and $C_2$. If $T$ intersects
either of these disks, then in either case it does so in simple closed curves
that either surround a single puncture, or surround both punctures. In the first
case, we can take an innermost such curve on $D _i$ and surger the torus to
obtain an essential annulus, both boundary-components of which become meridians
in the link complement. In the second case, we can take an outermost such curve
and surger $T$ to obtain an essential annulus, both boundary components of which
are longitudes on $\partial N(C_i)$. Thus, in either case $L$ is one of the
exceptions stated by the lemma.

Otherwise, if $T$ does not intersect $D_1 \cup D_2$, we can assume that the
torus is outside the ball $N(D_1 \cup D_2)$. Since $\partial N(D_1 \cup D_2)$
is a 4-punctured sphere isotopic to the 4-punctured sphere $\cQ$, the torus can
be pushed out of $\cB$ entirely. But then the incompressible torus $T$ lies
in $\cB'$ in the complement of the two unknotted arcs that make up the rational tangle.
If $T$ cuts a solid torus from $\cB'$, it will compress in that solid torus.
So it must be that the torus is a cube-with-knotted-hole. But then to prevent
compression in the knotted hole, the two arcs must form a knot that passes
through the hole, contradicting the fact this is a rational tangle, each arc of which is trivial.
\end{proof}

\bibliographystyle{plain}
\bibliography{bill}

\noindent DEPARTMENT OF MATHEMATICS, WILLIAMS COLLEGE, WILLIAMSTOWN, MA 01267, USA

\noindent{\it E-mail address:} cadams@williams.edu

\medskip

\noindent MATHEMATICS DEPARTMENT, UNIVERSITY OF MASSACHUSETTS, AMHERST, MA 01003, USA

\noindent {\it E-mail address:} profmeeks@gmail.com

\medskip

\noindent DEPARTMENTO DE MATEM\'ATICA PURA E APLICADA, UNINVERSIDADE FEDERAL DO RIO GRANDE
DO SUL, BRAZIL

\noindent {\it E-mail address:}alvaro.ramos@ufrgs.br 

\end{document}